\documentclass[11pt]{article}

\usepackage[english]{babel}
\usepackage{amsmath,amssymb,amsthm}
\usepackage{graphicx,color}
\usepackage{enumerate}

\newtheorem{theorem}{Theorem}[section]
\newtheorem{proposition}[theorem]{Proposition}
\newtheorem{lemma}[theorem]{Lemma}
\newtheorem{corollary}[theorem]{Corollary}

\newtheorem*{results}{Main Results}

\theoremstyle{definition}
\newtheorem{definition}[theorem]{Definition}

\newtheorem{remark}{Remark}
\newtheorem{assumption}{Assumption}

\newcommand{\R}{\mathbb{R}}

\renewcommand{\P}{\mathbb{P}}
\newcommand{\E}{\operatorname{\textbf{E}}}

\newcommand{\expect}[1]{\E \left[ #1 \right]}

\newcommand{\condexpect}[2]{\E \left[ \left. #1 \right| #2 \right]}

\newcommand{\indicate}[1]{\mathbf{1}  \left \{ #1 \right \}}

\newcommand{\mt}{\mathcal{T}}

\newcommand{\vsum}{\sum_{|v|=n}}

\renewcommand{\root}{\varsigma}

\newcommand{\hdim}{\operatorname{dim}_H}

\newcommand{\Gt}{\Gamma_t}
\newcommand{\Gtrl}{\Gamma_t(\root_L)}
\newcommand{\Gtrr}{\Gamma_t(\root_R)}

\begin{document}

\title{Diffusions of Multiplicative Cascades}

\author{\begin{tabular}{cc}
Tom Alberts & Ben Rifkind \\
\small{Department of Mathematics} & \small{Department of Mathematics} \\
\small{California Institute of Technology} & \small{University of Toronto} \\
\small{Pasadena, CA, USA} & \small{Toronto, ON, Canada} \\
\end{tabular}
\\}

\date{}

\maketitle

\begin{abstract}
A multiplicative cascade can be thought of as a randomization of a measure on the boundary of a tree, constructed from an iid collection of random variables attached to the tree vertices. Given an initial measure with certain regularity properties, we construct a continuous time, measure-valued process whose value at each time is a cascade of the initial one. We do this by replacing the random variables on the vertices with independent increment processes satisfying certain moment assumptions.  Our process has a Markov property: at any given time it is a cascade of the process at any earlier time by random variables that are independent of the past. It has the further advantage of being a martingale and, under certain extra conditions, it is also continuous. For Gaussian independent increments processes we develop the infinite-dimensional stochastic calculus that describes the evolution of the measure process, and use it to compute the optimal H\"{o}lder exponent in the Wasserstein distance on measures. We also discuss applications of this process to models of tree polymers and one-dimensional random geometry.
\end{abstract}

\let\thefootnote\relax

\let\thefootnote\relax

\section{Introduction \label{sec:intro}}

Multiplicative cascades are a particular type of random measures with many interesting statistical properties. The space on which these measures live is not always the same, but there is typically a tree structure underlying their construction and so it is convenient to consider them as living on the boundary of an infinite tree. This is the situation we consider. This has the further advantage that several different models of statistical mechanics are fully described by this framework, most notably tree polymers, branching random walk, and certain models of random walk in random environment.

For simplicity we work on a rooted, infinite binary tree $\mt$, and the boundary $\partial \mt$ is the set of all infinite self-avoiding paths that begin at the root. Elements of $\partial \mt$ are called rays and we denote them by $\xi$. The inputs to the cascade model are a positive measure $\Gamma$ on $\partial \mt$, which can be specified arbitrarily, and an i.i.d. collection of random variables $\{W(v) \}_{v \in \mt}$ attached to the vertices of the tree. The only a priori assumption on the distribution of the $W$ is that it is strictly positive and has mean one. These random variables are then cascaded on to $\Gamma$ to produce a random measure on $\partial \mt$; we denote it by $\Gamma_W$ or sometimes
\begin{align*}
\Gamma_W = \mathcal{C}(\Gamma; W).
\end{align*}
The cascading procedure is simple to describe: for each $n \geq 0$ one uses the random weights $W$ up to generation $n$ to construct a random measure via
\begin{align*}
d \Gamma_W^{(n)}(\xi) = \prod_{i=1}^n W(\xi_i) \, d \Gamma(\xi).
\end{align*}
The random cascade measure is then defined as the limit
\begin{align}\label{eq:limit_defn}
\Gamma_W := \lim_{n \to \infty} \Gamma_W^{(n)}.
\end{align}
A martingale argument shows that the limit exists almost surely for \textit{any} choice of the initial measure $\Gamma$, in the topology of weak convergence on the space of measures. Full details are given in Section 2. As we will see there it \textit{may} happen that $\Gamma_W$ is the zero measure, but nonetheless it is well-defined, and given this the main problem is to determine the properties of $\Gamma_W$ and how they depend on the input measure $\Gamma$ and the cascading distribution $W$. Fundamental properties of cascade measures were derived in \cite{KP76}, and further explorations have been made in several later papers; see for example \cite{Big77, HW92, LR00, OW02, F02}.

Even in the simplest cases the relationship between $\Gamma_W$ and $\Gamma$ is interesting. Observe that if $W = 1$ then $\Gamma_W = \Gamma$, but if the cascading distribution is not identically one then $\Gamma_W$ is necessarily distinct from $\Gamma$. There are two possible alternatives:
\begin{itemize}
\item $\Gamma_W$ may be identically the zero measure, even though $\Gamma$ is not, but
\item if $\Gamma_W$ is not the zero measure then it is genuinely random, meaning it depends on the specific realization of the $W$ variables, but almost surely it is singular with respect to $\Gamma$.
\end{itemize}
The positivity of $\Gamma_W$ is determined by both the regularity of $\Gamma$ (roughly meaning how strongly it concentrates on some rays more than others) and moment properties of the cascading distribution. Full details are given in Section \ref{sec:Background}. The singularity property, however, holds even if the cascading distribution is highly concentrated near one. It is a simple consequence of the fact that along any ray the density is the product of positive, iid, mean one random variables, which almost surely goes to zero as the number of terms in the product goes to infinity.

The main purpose of this paper is to study what happens when the cascading distribution is highly concentrated near one and the cascading procedure is iterated. The scheme is simple: start with a positive measure $\Gamma$ on $\partial \mt$ and cascade once to produce $\Gamma_W$. Since the cascading procedure does not depend on the choice of the initial measure, we may use $\Gamma_W$ as the input measure and cascade it with vertex variables $\{ W^*(v) \}_{v \in \mt}$ that are independent of the $\{ W(v) \}$ collection. This iteration can be repeated indefinitely, at each time cascading with a collection of vertex variables that are independent of all previous ones, and in doing so it produces a discrete time, measure-valued Markov process.

This discrete time process is interesting in its own right, but we prefer instead to study a continuous time version. Intuitively the idea behind the continuous time process is clear: starting from some initial measure, in each infinitesimal unit of time we cascade the previous measure with an independent collection of random variables whose distribution is an infinitesimal perturbation away from the degenerate distribution at one. Repeating this scheme indefinitely builds the process.

As is usual, however, rigorously constructing the continuous time process takes more care than constructing the discrete time one, even though the basic idea is the same. Several different construction techniques could be considered; for example, the discrete time process is well-defined, and the continuous time process could be constructed by taking a weak limit as the discrete time step goes to zero and the cascading distribution concentrates near one. Alternatively, the process is essentially defined by saying that the measure at each time is a cascade of the process at an earlier time (by an independent collection of random variables); this is akin to specifying the transition probabilities of the process, and then the existence would follow from the general theory on measure-valued diffusions (see for example \cite{EK86}).

In this paper we propose a simpler and more direct construction procedure. Instead of appealing to the more abstract concepts above, we simply attach to the vertices of the tree a family of dynamic weights $\{ t \mapsto W_t(v) \}_{v \in \mt}$. Using the cascading procedure defined in equation \eqref{eq:limit_defn}, this gives us a process $t \mapsto \Gamma_t := \Gamma_{W_t}$ of random cascade measures. We choose the weight process $t \mapsto W_t$ so that the $\Gamma_t$ process satisfies the following important Markov property: the value at any given time is a cascade of the value at any previous time, by a noise that is independent of the past of the process. More precisely, our process is defined on an interval $[0,T]$ for some $T > 0$, and has the property that for any $s, t \geq 0$ such that $t+s \leq T$, \textit{both} of the relations
\begin{align*}
\Gamma_{t+s} = \mathcal{C}(\Gamma; W_{t+s}) \quad \mathrm{ and } \quad \Gamma_{t+s}  = \mathcal{C} \left( \mathcal{C}\left( \Gamma; W_t \right); \frac{W_{t+s}}{W_t} \right)
\end{align*}
hold. This is a fully rigorous statement, but should be regarded as a manifestation of the non-rigorous infinitesimal cascading procedure described earlier. The main focus of our paper is to show that, under suitable assumptions on the i.i.d. collection of weight processes $W_t(v)$ attached to the vertices of the tree, the following is true:

\begin{results}
Assume that the process $t \mapsto \log W_t$ is an independent increments process, with $W_0 = 1$, $\expect{W_t} = 1$, and $W_t$ always strictly positive. Assume the process is defined on an interval $[0,T]$ for some $T > 0$. If there is a $\delta > 0$ such that $W_T$ has a finite $(1+\delta)$ moment, and the measure $\Gamma$ is $W_T$-regular (see Definition \ref{defn:regularity}), then
\begin{itemize}
\item the process $\Gamma_t := \mathcal{C} (\Gamma, W_t)$ is well-defined on $[0,T]$, i.e. the event that $\lim_{n \to \infty} \Gamma_t^{(n)}$ exists for all $0 \leq t \leq T$ has full probability,
\item for any $s, t \geq 0$ with $t+s \leq T$, the equality $\Gamma_{t+s} = \mathcal{C}(\Gamma_t, W_{t+s}/W_t)$ also holds almost surely,
\item the process is a martingale with respect to the filtration $\sigma \left( \Gamma_s : s \leq t \right)$,
\item if the process $t \mapsto W_t$ is continuous, then so is the $\Gamma_t$ process in the topology of weak convergence of measures.
\item for $t \mapsto \log W_t$ a Gaussian process the measure process $\Gamma_t$ is H\"{o}lder-($\frac12 - \epsilon$) continuous in the Wasserstein metric on measures for any $\epsilon > 0$, but not H\"{o}lder-($\frac12 + \epsilon$) continuous.
\end{itemize}
\end{results}

These results are intuitive, but we want to emphasize that they are not immediate. It is easily seen that all four of these properties hold trivially for the finite level $t \mapsto \Gamma_t^{(n)}$ processes, but it requires some extra work to carry them over to the limit as $n \to \infty$.  For fixed $t$ and $s$, the Markov property, which is essentially a result about the composition of cascades, was first proven by \cite{WW95} and later reproved in \cite{F}. The existence of a discrete time Markov process would therefore follow from their work. With somewhat different analysis, we take care of the subtle difficulties in extending this notion to a continuous time process.  The main technical difficulty is that the process cannot be started from an arbitrary measure; it has to be started from those which enjoy a sufficient amount of regularity. For practical applications the regularity condition we use is not at all restrictive, but we have to ensure that once the process begins it will stay within the class of sufficiently regular measures so that it can be continued. In Section \ref{sec:Background} we describe exactly what we mean by sufficiently regular, and in Section \ref{sec:process} we prove that the evolution of the regularity of the process is well-behaved. This is a part of our proof of the results above.

It is also important to note that our main technique of replacing static weights with time varying processes has already been carried out for several other models. Likely the most prominent one is Dyson's Brownian motion, which is obtained by replacing the Gaussian entries of the GUE matrices with standard Brownian motions. More recently, however, the idea has been applied to the Sherrington-Kirkpatrick model of spin glasses in \cite{CN95}, and then re-applied to greater effect by a series of other authors \cite{BKL02, Tin05}. The paper \cite{MRT11} also uses the same technique in the context of lattice polymer models, which are somewhat similar to ours through the connection between tree polymers and multiplicative cascades. However, the main purpose of these papers is to use the dynamic weights technique to derive growth exponents and fluctuation behavior for partition functions of Gibbs measures as the size of the system grows large, whereas we are more concerned with showing that the infinite volume measure-valued process has the properties listed above.

We put particular emphasis on the results derived in Sections \ref{sec:gaussian} and \ref{sec:Holder}, where we specialize to the case when $\log W_t$ is a Brownian motion. This allows us to extend classical stochastic calculus results to this infinite-dimensional setting and use them to describe the evolution of the measures. One of our long term goals is to use these stochastic calculus techniques to compute explicit formulas for probability densities of certain quantities related to the measure; for example the total mass at any fixed time. We believe this is possible, but ultimately it will require more refined techniques that are beyond the scope of the current paper. Nonetheless, interesting results can already be derived using the stochastic calculus that we develop, and in Section \ref{sec:Holder} we use it to show H\"{o}lder continuity of the measure process in the Wasserstein distance. We also show that the optimal H\"{o}lder exponent is $1/2$. Both are somewhat surprising facts, since for any $t \neq s$ the measures $\Gamma_t$ and $\Gamma_s$ are almost surely singular; hence the process $t \mapsto \Gamma_t$ is very discontinuous in the total variation distance. Given this it is not immediately clear that continuity can be expected in any topology stronger than the one induced by weak convergence, and our result should be viewed in this context.

The outline of this paper is as follows: in Section \ref{sec:Background} we set up our notation and recall some well known properties of cascade measures. In Section \ref{sec:process} we construct the process and show that it is well-defined, and give proofs for the main results listed above. In Section \ref{sec:gaussian} we discuss the special case when the weight process is an exponential of a Brownian motion, and use stochastic calculus to describe the infinitesimal evolution of the process. This shows one advantage of our construction over the more abstract possibilities listed earlier: it allows for a full description of the evolution of the measure-valued process in terms of the input weight process $t \mapsto W_t(v)$. In Section \ref{sec:applications} we describe possible applications of our process to models of tree polymers and to the KPZ formula of one-dimensional random geometry.\newline

\textbf{Acknowledgements:} We thank B\'{a}lint Vir\'{a}g for several helpful comments and suggestions, and Sourav Chatterjee for pointing out the connection with \cite{CN95}.

\section{Background and Notation \label{sec:Background}}

We begin with our notation for trees. Let $\mt$ be a rooted infinite binary tree and denote the root by $\root$. Given a vertex $v \in \mt$ we let $|v|$ be its generation, by which we mean its distance from the root. Let $v_L$ and $v_R$ be the left and right offspring of $v$, respectively, and write $v_p$ for the parent of $v$. We let $\mt(v)$ be the subtree of $\mt$ rooted at $v$. Note that when working on subtrees we still use $|u|$ to denote the distance from $\root$, not from the root of the subtree.

We will mostly be working on the boundary of $\mt$, which we denote by $\partial \mt$. Recall that $\partial \mt$ is the set of all infinite self-avoiding paths in the tree that begin at the root. Elements of $\partial \mt$ are called rays and are usually denoted by $\xi$. We denote by $\xi_n$ the vertex in the $n^{th}$ generation of the ray $\xi$. Given two rays $\xi$ and $\zeta$ we let $\xi \wedge \zeta$ be the vertex of $\mt$ that is the last common ancestor of $\xi$ and $\zeta$. For a given vertex $v \in \mt$ we let $\partial \mt(v)$ be the set of all rays passing through $v$.

\subsection{Measures on $\partial \mt$ \label{sec:general_measures}}

Even though $\partial \mt$ is an uncountable set, a measure on $\partial \mt$ is completely determined by the countable collection of values $\{ \Gamma(\partial \mt(v)) \}_{v \in \mt}$. Hence every positive, finite measure on $\partial \mt$ can be identified with a function $\Gamma : \mt \to \R_+$ satisfying the two conditions
\begin{itemize}
\item $0 < \Gamma(\root) < \infty$,
\item for every vertex $v \in \mt$, $\Gamma(v) = \Gamma(v_L) + \Gamma(v_R).$
\end{itemize}
Due to this identification, measures on $\partial \mt$ are also called flows on $\mt$. As long as $0 < \Gamma(\root) < \infty$ it is possible to normalize $\Gamma$ to be a probability measure, i.e. so that $\Gamma(\root) = 1$. Sometimes we denote the normalized measure by $\Gamma^*$, but in general we do not assume that we are working with probability measures.

A special measure on $\partial \mt$ is the ``Lebesgue'' measure given by $\theta(v) = 2^{-|v|}$. Observe that $\theta$ is also the measure induced on $\partial \mt$ by constructing random paths via simple random walk; that is, starting at the root and then using independent and unbiased coin flips at each vertex to decide whether to move left or right down the tree.

The topology on measures is as follows: we say that a sequence of measures $\Gamma_n$ converges to $\Gamma$ if $\Gamma_n(v) \to \Gamma(v)$ for all $v \in \mt$. This is equivalent to weak convergence of $\Gamma_n \to \Gamma$, when the topology on $\partial \mt$ is generated by the metric
\begin{align*}
d(\xi, \eta) = \theta(\xi \wedge \eta) = 2^{-|\xi \wedge \eta|}.
\end{align*}

For a vertex $v \in \mt$ we will write $\Gamma_{|v}$ for the measure restricted to the subtree $\mt(v)$.

\subsection{Random Cascade Measures on $\partial \mt$}

In this section we describe how to take a measure $\Gamma$ on $\partial \mt$ and a collection of random variables to construct a cascade measure. Let $W$ be a random variable that is positive almost surely and has mean one. We are mostly concerned with its distribution which we call the \textit{cascading distribution}. Assume that $W$ is not identically one, and therefore Jensen's inequality implies that $\expect{\log W} < 0$.

Let $\{W(v)\}_{v \in \mt}$ be a collection of i.i.d. random variables with common distribution $W$. From this collection we build a random function $X : \mt \to \R_+$ defined by
\begin{align*}
X(\xi_n) = \prod_{i=1}^n W(\xi_i).
\end{align*}
Then for each $n \geq 0$ we construct a random measure $\Gamma_W^{(n)}$ by specifying the Radon-Nikodym derivative with respect to $\Gamma$ as
\begin{align*}
d \Gamma_W^{(n)}(\xi) := X(\xi_n) \, d \Gamma(\xi).
\end{align*}
The random cascade measure is then defined as the limit of $\Gamma_W^{(n)}$ as $n \to \infty$. Recall that the topology is pointwise in the vertices, meaning that
\begin{align}\label{cascade_defn}
\Gamma_W(v) = \lim_{n \to \infty} \Gamma_W^{(n)}(v)
\end{align}
for every $v \in \mt$. A simple martingale argument, which we now recall, shows that the limit always exists. First consider the case $v = \root$, so that
\begin{align*}
\Gamma_W^{(n)}(\root) = \int_{\partial \mt} X(\xi_n) d \Gamma(\xi).
\end{align*}
It is easy to see that $X(\xi_n)$ is a martingale with respect to the filtration
\begin{align*}
\mathcal{W}_n := \sigma(W(v): |v| \leq n),
\end{align*}
and therefore so is $\Gamma_W^{(n)}(\root)$ by Fubini's Theorem. Since $\Gamma_W^{(n)}(\root)$ is positive it converges almost surely, and since positivity of the limit does not depend on any finite collection of the $W(v)$ variables a standard $0$-$1$ law argument shows that the limit is almost surely zero or almost surely strictly positive. In the case that $\Gamma = \theta$ Kahane and Peyriere \cite{KP76} showed that
\begin{align*}
\expect{W \log W} < \log 2
\end{align*}
is a necessary and sufficient condition for  the limit to be positive. In the case of a general measure $\Gamma$ it remains an open problem to determine sharp criterion for when $\Gamma_W(\root) > 0$, but there are many known sufficient conditions involving moment of $W$ and the regularity of $\Gamma$. We will use a condition of Fan \cite{F02}, for which we need the following definitions:

\begin{definition}\label{defn:regularity}
For $\Gamma$ a measure on $\partial \mt$, define the \textit{pressure function} $\lambda_{\Gamma} : [0, \infty) \to \R$ by
\begin{align*}
\lambda_{\Gamma}(h) := \limsup_{n \to \infty} \frac{1}{n} \log \sum_{|v|=n} \Gamma(v)^h.
\end{align*}
We will say that a measure $\Gamma$ on $\partial \mt$ is $W$-regular if
\begin{align*}
\expect{W \log W} + \lambda_{\Gamma}'(1+) < 0.
\end{align*}
We say that it is $W$-irregular if
\begin{align*}
\expect{W \log W} + \lambda_{\Gamma}'(1-) > 0.
\end{align*}
\end{definition}

Observe that $\lambda_{\Gamma}(1) = 0$ for any $\Gamma$. Fan \cite{F02} uses the pressure function to derive the following condition:

\begin{proposition}[\cite{F02}]
Suppose there exists a $\delta > 0$ with $\expect{W^{1+\delta}} < \infty$ for some $\delta > 0$. Then
\begin{itemize}
\item if $\Gamma$ is $W$-regular then $\Gamma_W(\root) > 0$ almost surely,
\item if $\Gamma$ is $W$-irregular then $\Gamma_W(\root) = 0$ almost surely.
\end{itemize}
\end{proposition}

Observe that if $\lambda_{\Gamma}$ is differentiable at $h=1$ then the condition of $W$-regularity is close to sharp. For $\Gamma = \theta$ we have $\lambda_{\theta}(h) = (1-h) \log 2$, and hence the condition of Kahane and Peyriere is recovered.

\begin{remark} \label{remark:inherit_regular}
Let $W_1$ and $W_2$ be two distinct cascading distributions, and suppose there is an $\epsilon > 0$ such that $\expect{W_1^h} \leq \expect{W_2^h} < \infty$ for $h \in [1, 1+\epsilon]$. Then since
\begin{align*}
\expect{W \log W} = \lim_{h \downarrow 0} \frac{\expect{W^h} - 1}{h}
\end{align*}
it follows that $\expect{W_1 \log W_1} \leq \expect{W_2 \log W_2}$. Hence $W_2$-regularity of $\Gamma$ implies $W_1$-regularity of $\Gamma$.
\end{remark}

\begin{remark} \label{remark:atom_free}
The assumption of $W$-regularity implicitly means that $\lambda_{\Gamma}$ is differentiable from the right at $h=1$.
\end{remark}

\begin{remark} \label{remark:submeasure_regularity}
It is important to note that $W$-regularity of a measure is a property that is inherited by all of its submeasures.  Indeed, since $\lambda_{\Gamma}$ is computed over a larger sum than $\lambda_{\Gamma_{|v}}$, it follows that $\lambda_{\Gamma}(h) \geq \lambda_{\Gamma_{|v}}(h)$ for all $h$. But also $\lambda_{\Gamma}(1) = \lambda_{\Gamma_{|v}}(1) = 0$, and therefore by the Mean Value Theorem
\begin{align*}
0 \leq \lambda_{\Gamma}(1+\epsilon) - \lambda_{\Gamma_{|v}}(1+\epsilon) = \lambda_{\Gamma}'(1+s) - \lambda_{\Gamma_{|v}}'(1+s)
\end{align*}
for some $s \in (0,\epsilon)$. Taking $\epsilon$ to zero gives $$\lambda_{\Gamma_{|v}}'(1+) \leq \lambda_{\Gamma}'(1+),$$ which implies $W$-regularity of $\Gamma_{|v}$.
\end{remark}

\begin{remark} \label{remark:submeasure_cascade}
We have only shown existence of the limit \eqref{cascade_defn} in the $v = \root$ case, and it is important to note that this only required that $\Gamma$ is a positive measure on $\partial \mt$. The regularity of $\Gamma$ determines whether the limit is positive or zero. But these facts and the self-similarity of the tree also combine to give us the existence and positivity of the limit for $v \neq \root$. Indeed, assume $n > |v|$, so that
\begin{align}
\Gamma_W^{(n)}(v) &= \int_{\partial \mt} X(\xi_n) \indicate{\xi \in \partial \mt(v)} d \Gamma(\xi) \notag \\
&= X(v) \int_{\partial \mt(v)} \frac{X(\xi_{n-|v|})}{X(v)} d \Gamma_{|v}(\xi). \label{eq:level_n_cascade}
\end{align}
But the integral term is just the level $n - |v|$ cascade of the $\Gamma_{|v}$ measure by the random variables $W_{|v} = \{ W(u) : u \in \mt(v) \}$, hence the martingale argument for the $v = \rho$ case also shows that its limit exists as $n \to \infty$. Its positivity can again be determined by Fan's condition, and by the last remark $W$-regularity is inherited by all submeasures. Thus if $\Gamma$ is $W$-regular then $\Gamma_W(v) > 0$ for all $v \in \mt$ with $\Gamma(v) > 0$. Taking the limit as $n \to \infty$ in equation \eqref{eq:level_n_cascade} gives the relation
\begin{align}\label{eq:subtree_mass}
\frac{\Gamma_W(u)}{X(v)} = \mathcal{C} \left( \Gamma_{|v}; W_{|v} \right)(u)
\end{align}
for all $u \in \mt(v)$.
\end{remark}

Finally we remark that even though the limits in \eqref{cascade_defn} are defined pointwise at each vertex, the limiting object $\Gamma_W$ is automatically a measure on $\partial \mt$. This follows from the definition of the level $n$ cascade as a measure, and therefore
\begin{align*}
\Gamma_W^{(n)}(v) = \Gamma_W^{(n)}(v_L) + \Gamma_W^{(n)}(v_R).
\end{align*}
Now take limits as $n \to \infty$.

\subsection{Rates of Convergence for the Cascading Procedure}

Our analysis in this section relies on that in [\cite{F02}]. In particular we will need to assume that the cascade variable $W$ satisfies a moment constraint and that the measure $\Gamma$ is $W$-regular.

\begin{assumption}\label{assumption:exp_decay}
We assume that
\begin{itemize}
\item There is a $\delta>0$ such that $\expect{W^{1+\delta}} < \infty$.
\item The measure $\Gamma$ is $W$-regular.
\end{itemize}
\end{assumption}

These assumptions allow for exponential control on the decay of the cascade measure.

\begin{definition} \label{remark:exp_decay}
Define
\begin{align*}
\alpha(h) := \alpha(h; W, \Gamma) = \limsup_{n \to \infty} \frac{1}{n} \log \sum_{|v|=n} \Gamma(v)^h \expect{X(v)^h} = \lambda_{\Gamma}(h) + \log \expect{W^h}.
\end{align*}
The moment assumption on $W$ implies that $\alpha(h) < \infty$ for $h$ in a neighborhood of $1$. Since $\alpha(1) = 0$ it is straightforward to compute that $\Gamma$ being $W$-regular implies that $\alpha'(1+) < 0$, and therefore $\alpha(1+ \epsilon) < \alpha(1) = 0$ for $\epsilon$ sufficiently small. Therefore we also define
$$h_W := \sup \{ h \geq 1 : \alpha(h) < 0 \},$$
and by the last remarks we have $h_W >1$ under Assumption \ref{assumption:exp_decay}.
\end{definition}

Much of our analysis will rely on having a rate of convergence of $\Gamma_W^{(n)}$ to $\Gamma_W$. We will heavily make use of the following lemma:

\begin{lemma}\label{lemma:rate}
For $1 \leq h \leq 2$ there is a positive constant $C = C(h)$ such that
\begin{align*}
\left| \left| \Gamma_W^{(n+1)}(\root) - \Gamma_W^{(n)}(\root) \right| \right|_{L^h} \leq C ||W||_{L^h}^{n+1} \left( \sum_{|v| = n} \Gamma(v)^h \right)^{1/h}  \!\!\!\!\!\! ,
\end{align*}
and therefore by the triangle inequality
\begin{align*}
\left| \left| \Gamma_W(\root) - \Gamma_W^{(n)}(\root) \right| \right|_{L^h} \leq C \sum_{m>n} ||W||_{L^h}^m \left( \sum_{|v| = m} \Gamma(v)^h \right)^{1/h} \!\!\!\!\!\!.
\end{align*}
\end{lemma}

The proof relies on the following inequality of von Bahr and Esseen:

\begin{lemma}[\cite{BE65}] \label{lemma:bahr_esseen}
Let $\{U_i \}$ and $\{V_i \}$ be sequences of random variables that are independent of each other. Also assume that the $\{ V_i \}$ are mutually independent, and that $\expect{V_i} = 0$ for all $i$. Then for $1 \leq h \leq 2$ there is a universal constant $c = c(h)$ such that
\begin{align*}
\expect{\left( \sum_i U_i V_i \right)^h} \leq c \sum_{i} \expect{U_i^h} \expect{V_i^h}.
\end{align*}
\end{lemma}

\begin{proof}[Proof of Lemma \ref{lemma:rate}]
We have the trivial identity
\begin{align*}
\Gamma^{(n+1)}(\root) - \Gamma^{(n)}(\root) &= \int \left( X(\xi_{m+1}) - X(\xi_m) \right) \, d \Gamma(\xi) \\
&= \int X(\xi_m) (W(\xi_{m+1}) - 1) \, d \Gamma(\xi) \\
&= \!\!\!\! \sum_{|v| = m+1} \!\!\!\! \Gamma(v) X(v_p) (W(v) - 1).
\end{align*}
The von Bahr-Esseen inequality applies to the latter sum, and therefore
\begin{align*}
\expect{\left| \Gamma_W^{(n+1)}(\root) - \Gamma_W^{(n)}(\root) \right|^h} &\leq c(h) \!\!\! \sum_{|v| = n+1} \!\!\! \Gamma(v)^h \expect{W^h}^{n} \expect{|W-1|^h} \\
& \leq 2 c(h) \expect{W^h}^{n+1} \!\!\! \sum_{|v| = n+1} \!\! \Gamma(v)^h.
\end{align*}
\end{proof}

The next lemma implies the $L^h$ convergence of the total mass of the cascade measure, and therefore that $\Gamma_W(\root) > 0$ almost surely.

\begin{corollary}\label{corollary:Lp_convergence}
Under Assumption \ref{assumption:exp_decay} we have that for all $h \in (1, h_W)$,
\begin{align*}
 \limsup_{n \to \infty} \frac{1}{n} \log \expect{|\Gamma_W(\root)-\Gamma_W^{(n)}(\root)|^h} \leq \lambda_{\Gamma}(h) + \log \E[W^h] < 0.
\end{align*}
\end{corollary}

\begin{proof}
Since $\Gamma$ is $W$-regular, by Definition \ref{remark:exp_decay} for all $h \in (1, h_W)$, $\alpha(h) = \lambda_{\Gamma}(h) + \log \expect{W^h} < 0$. Hence for each $\gamma > 0$ such that $\alpha(h) + \gamma < 0$ there is a positive constant $C$ such that
\begin{align*}
\sum_{|v|=n} \Gamma(v)^h \expect{X(v)^h} \leq C e^{n(\alpha(h) + \gamma)}
\end{align*}
for all $n$. Applying the second statement of Lemma \ref{lemma:rate} completes the proof.
\end{proof}

We now extend Corollary \ref{corollary:Lp_convergence} to show that the exponential rate of convergence is uniform for all vertices on a fixed generation of the tree.

\begin{lemma} \label{lemma: cascade_Fan}
Under Assumption \ref{assumption:exp_decay} we have that for $h \in (1, h_W)$ and for $1 \leq i \leq n$
$$
\limsup_{n \to \infty} \frac{1}{n} \log \sum_{|v|=i} \expect{\left| \Gamma_W(v) - \Gamma_W^{(n)}(v) \right|^h}
\leq \lambda_{\Gamma}(h) + \log \expect{W^h} < 0.
$$
And moreover,
\begin{align*}
\limsup_{n \to \infty} \frac{1}{n} \log \sum_{i=1}^n \sum_{|v|=i} \expect{\left| \Gamma_W(v) - \Gamma_W^{(n)}(v) \right|^h} \leq \lambda_{\Gamma}(h) + \log \E[W^h] < 0.
\end{align*}
\end{lemma}

\begin{proof}
From \eqref{eq:level_n_cascade} we have, for $|v|=i \leq n$,
\begin{align*}
\Gamma_W^{(n)}(v) = X(v) \mathcal{C} \left( \Gamma_{|v}; W_{|v} \right)^{(n-i)}(v).
\end{align*}
Combining this with \eqref{eq:subtree_mass} and using that $X(v)$ is independent of the cascade on $\mt(v)$ gives
\begin{align*}
\expect{|\Gamma_W(v) - \Gamma_W^{(n)}(v)|^h} = \expect{X(v)^h} \expect{ \left| \mathcal{C}(\Gamma_{|v}; W_{|v})(v) - \mathcal{C}(\Gamma_{|v}; W_{|v})^{(n-i)}(v) \right|^h}.
\end{align*}
Applying Lemma \ref{lemma:rate} to the second factor and combining with the first factor gives
\begin{align}\label{eq:middle_fan_bound}
\expect{|\Gamma_W(v) - \Gamma_W^{(n)}(v)|^h} \leq C \left( \sum_{k > n-i} \left( \sum_{\substack{w \in \mt(v) \\ |w|=|v|+k}} \Gamma(w)^h \expect{W^h}^{|v|+k} \right)^{1/h} \right)^h,
\end{align}
where $C$ depends only on $h$. Now define $a_k(v)$ by
\begin{align*}
a_k(v) = \!\!\!\! \sum_{\substack{w \in \mt(v) \\ |w| = |v|+k}} \!\!\!\! \Gamma(w)^h \expect{W^h}^{|v|+k}
\end{align*}
and $\mathbf{a_n}(v) = (a_{n+1}(v), a_{n+2}(v), a_{n+3}(v), \ldots)$. Then equation \eqref{eq:middle_fan_bound} is equivalent to
\begin{align*}
\expect{|\Gamma_W(v) - \Gamma_W^{(n)}(v)|^h} \leq C || \mathbf{a_{n-i}}(v) ||_{\ell^{1/h}},
\end{align*}
with $\ell^{1/h}$ denoting the usual sequence space. Summing over $|v|=i$ gives
\begin{align}\label{second_fan_bound}
\sum_{|v|=i} \expect{|\Gamma_W(v) - \Gamma_W^{(n)}(v)|^h} \leq C \sum_{|v|=i} || \mathbf{a_{n-i}}(v) ||_{\ell^{1/h}} \leq C \left| \left| \sum_{|v|=i} \mathbf{a_{n-i}}(v) \right| \right| _{\ell^{1/h}} \!\!\!\!\!\!\!\! .
\end{align}
The last inequality is the Minkowski inequality for $\ell^{1/h}$ (recall $h \geq 1$). By definition of $\mathbf{a}$ we have
\begin{align}\label{eq:third_fan_bound}
\sum_{|v|=i} \mathbf{a_{n-i}}(v) = (a_{n+1}(\root), a_{n+2}(\root), a_{n+3}(\root), \ldots) = \mathbf{a_{n}}(\root).
\end{align}
By definition of $\alpha(h)$ we have, for each $\gamma > 0$,
\begin{align*}
a_n(\root) = \sum_{|v|=n} \Gamma(v)^h \expect{W^h}^n \leq e^{n(\alpha(h) + \gamma)}
\end{align*}
for $n$ sufficiently large. Under Assumption \ref{assumption:exp_decay} and using Remark \ref{remark:exp_decay} we have $\alpha(h) < 0$ for $h \in (1, h_W)$. Choosing $\gamma$ such that $\alpha(h) + \gamma < 0$, this gives
\begin{align*}
|| \mathbf{a_n}(\root) ||_{\ell^{1/h}} \leq C e^{n (\alpha(h) + \gamma)}
\end{align*}
for $n$ sufficiently large. Combining this with \eqref{second_fan_bound} and \eqref{eq:third_fan_bound} and sending $\gamma$ to zero gives the first statement of the lemma. For the second part, simply observe that by \eqref{second_fan_bound} we have
\begin{align*}
\sum_{i=1}^n \sum_{|v|=i} \expect{|\Gamma_W(v) - \Gamma_W^{(n)}(v)|^h} \leq C n || \mathbf{a_n}(\root) ||_{\ell^{1/h}}.
\end{align*}
\end{proof}

This easily implies a uniform control of some moment of the cascade measure over all the vertices $v$ in the tree.

\begin{corollary} \label{cascade_Moment}
Under Assumption \ref{assumption:exp_decay} we have that for all $h \in (1, h_W)$
\begin{flalign*}
\limsup_{n \to \infty} \frac{1}{n} \log \sum_{|v|=n} \expect{\left| \Gamma_W(v) \right|^h}
& \leq \lambda_{\Gamma}(h) + \log \expect{W^h} < 0
\end{flalign*}

\end{corollary}
\begin{proof}
From the inequality $|a+b|^h \leq 2^h(|a|^h + |b|^h)$ we have
\begin{align*}
\expect{\Gamma_W(v)^h} \leq 2^h \left( \expect{\left| \Gamma_W(v) - X(v) \Gamma(v) \right|^h } + \Gamma(v)^h \expect{ X(v)^h } \right).
\end{align*}
Summing over $|v|=n$ and taking logarithms we get
\begin{align*}
\limsup_{n \to \infty} \frac{1}{n} \log \vsum \expect{\Gamma_W(v)^h}
& \leq \limsup_{n \to \infty} \frac{1}{n} \log \vsum \left( \expect{\left| \Gamma_W(v) - X(v) \Gamma(v) \right|^h } + \Gamma(v)^h \expect{ X(v)^h } \right) \\
& \leq \lambda_{\Gamma}(h) + \log \expect{W^h}.
\end{align*}
The last inequality is a consequence of the fact that the two terms in the line above both have the same exponential rate of decay, which is itself a consequence of Lemma \ref{lemma: cascade_Fan}.
\end{proof}

\section{A Markovian Random Cascade Process \label{sec:process}}

\subsection{Dynamic Random Weights \label{sec:weights}}

The main idea of this paper is to replace the random weights $W$ on the vertices of the tree with random weight processes $t \mapsto W_t$ that evolve in time. As usual we require a moment of the cascade variable $W_t$ as well as regularity of the measure $\Gamma$. To this we also add an independence condition. Throughout we assume the following properties of the weight processes and the initial measure $\Gamma$.

\begin{assumption}\label{assumption:basic} %\label{assumption:basic}
The weight process $t \mapsto W_t$ is defined in an interval $[0,T]$ with $T > 0$, and
\begin{itemize}
\item there is a $\delta>0$ such that $\expect{W_T^{1+\delta}} < \infty$,
\item the measure $\Gamma$ is $W_T$-regular.
\item $W_0 = 1$,
\item $W_t > 0$ and $\expect{W_t} = 1$ for each $t \geq 0$,
\item $t \mapsto \log W_t$ has independent increments.
\end{itemize}
\end{assumption}

\begin{remark}\label{remark:increasing_moments}
Observe that for $p>1$
\begin{align*}
\expect{W_T^{p} } = \expect{W_t^p} \E\left[ \frac{ W_{T}}{W_t} \right]^p \geq \expect{W_t^p},
\end{align*}
and so, by Remark \ref{remark:inherit_regular}, the moment and regularity assumptions are inherited for $W_t$ with $t < T$ .
\end{remark}

Such processes are easy to construct, for example exponentials of Brownian motion or exponentials of Levy processes (both properly normalized so that $\expect{W_t} = 1$). Note, however, that in both of these examples the increments of $\log W_t$ are stationary, but that our results do \textit{not} require this. For $s, t \geq 0$ we define
\begin{align*}
W_{t,t+s} := \frac{W_{t+s}}{W_t}.
\end{align*}
The independent increments assumption gives that $W_{t,t+s}$ is independent of $W_t$. Moreover the process $t \mapsto W_t$ is a martingale, that is
\begin{align*}
\expect{W_t | \sigma(W_r : r \leq s)} = W_s.
\end{align*}

Now to each vertex $v \in \mt$ attach a copy $W_t(v)$ of this process such that the collection $\{ W_t(v) \}_{v \in \mt}$ is independent. The main idea of this paper is to use the cascading procedure of the last section to construct a random cascade measure $\Gamma_{W_t}$ at each time $t \geq 0$, and then show that the resulting process $t \mapsto \Gamma_{W_t}$ is Markovian. This is carried out in Section \ref{sec:markov_property}, and the rest of the paper studies properties of the process. To simplify notation we write
\begin{align*}
\Gamma_t := \Gamma_{W_t} = \mathcal{C}(\Gamma; W_t).
\end{align*}
Observe that $\Gamma_0 = \Gamma$. We also define functions $X_{t, t+s} : \mt \to \R_+$ by
\begin{align*}
X_{t,t+s}(\xi_n) = \prod_{i=1}^n W_{t,t+s}(\xi_i),
\end{align*}
and the filtrations
\begin{align*}
\mathcal{W}_t = \sigma \left( W_s(v) : v \in \mt, s \leq t \right), \quad
\mathcal{F}_t = \sigma \left( \Gamma_s(v) : v \in \mt, s \leq t \right).
\end{align*}
In general $\mathcal{F}_t \subset \mathcal{W}_t$ and the inclusion in strict, since by knowing the weights one can construct the measure, but knowing the measure does not generally give full information on the weights.

To simplify notation, we will often drop references to the weights $W_t$ or the initial measure $\Gamma$.

\begin{definition} \label{h_time}
For $t \in [0,T]$, let
\[
\alpha_t(h) := \limsup_{n \to \infty} \frac{1}{n} \log \sum_{|v|=n} \Gamma(v)^h \expect{X_t(v)^h} = \lambda_{\Gamma}(h) + \log \expect{W_t^h}.
\]
Furthermore let
\[
h_t := \sup \{ h \geq 1 \text{ : such that } \alpha_t(h) < 0 \}.
\]
\end{definition}
\begin{remark} \label{remark:h_time}
Under Assumption \ref{assumption:basic}, $h_t >1$ for $t \in [0,T]$. Further note that $h_t$ is decreasing with $t$.
\end{remark}

\subsection{Construction and Basic Properties of the Process}

Before proving that the $t \mapsto \Gamma_t$ process is Markov we first deal with a technical issue. Above we said that we construct the process $t \mapsto \Gamma_t$ by applying the random cascading procedure of Section 2 at each fixed time $t$. However the existence of the random cascade measure is only an almost sure statement, and the event that it does not exist could conceivably depend on $t$. Since we are now working in continuous time it is possible that there is an exceptional set of times for which the cascade does \textit{not} exist, which would leave our cascade process ill-defined. We begin by showing that this is not the case.

To this end first note that for each $n > 0$ the finite level measure processes $t \mapsto \Gamma_t^{(n)}$ are well-defined, and in fact are martingales in $t$ with respect to the filtration $\mathcal{W}_t$. Indeed
\begin{align*}
\expect{ d\Gamma_{t+s}^{(n)}(\xi) | \mathcal{W}_t } = X_t(\xi_n) d \Gamma (\xi) \expect{X_{t,t+s}(\xi_n)} = d \Gamma_{t}^{(n)}(\xi),
\end{align*}
by the fact that $X_{t,t+s}$ is independent of $\mathcal{W}_t$ and has mean one. We will show that this martingale property, together with the exponential $L^p$ convergence of the finite level measures, gives that the $\Gamma_t$ process is well-defined. Moreover, the martingale property of the finite level measures is inherited by the limit.

\begin{theorem} \label{theorem:existence}
Under Assumptions \ref{assumption:basic}, the event
\begin{align*}
\left \{ \lim_{n \to \infty} \Gamma_t^{(n)}(v) \textrm{ exists for all } v \in \mt, t \leq T \right \}
\end{align*}
has probability one. Moreover,
\begin{enumerate}[(i)]
\item for each $v \in \mt$ the process $t \mapsto \Gamma_t(v)$ is a martingale with respect to $\mathcal{W}_t$, and hence $\mathcal{F}_t$, and,
\item if the weight process $t \mapsto W_t$ is continuous then so  is $\Gamma_t(v)$ for each $v \in \mt$.
\end{enumerate}
\end{theorem}

\begin{proof}
We concentrate first on the case $v = \root$. Fix $h \in (1, h_T)$. Since the difference $\Gamma_t^{(n+1)}(\root) - \Gamma_t^{(n)}(\root)$ is a martingale in $t$ (with respect to the filtration $\mathcal{W}_t$), we may apply Doob's maximal $L^h$ inequality to get that
\begin{flalign*}
P \left( \sup_{0 \leq t \leq T} |\Gamma_t^{(n+1)}(\root) - \Gamma_t^{(n)}(\root)| > \beta^n \right)
& \leq \beta^{-nh} \expect{\sup_{0 \leq t \leq T} |\Gamma_t^{(n+1)}(\root) - \Gamma_t^{(n)}(\root)|^h} \\
& \leq \beta^{-nh} \left( \frac{h}{h-1} \right)^h \expect{|\Gamma_T^{(n+1)}(\root) - \Gamma_T^{(n)}(\root)|^h} \\
& \leq C\beta^{-nh} \expect{W_T^h}^{n} \sum_{|v|=n} \Gamma(v)^h,
\end{flalign*}
with the last inequality coming from Lemma \ref{lemma:rate}. Therefore by taking logarithms we get
$$
\limsup_{n \to \infty} \frac{1}{n}\ \log P \left( \sup_{0 \leq t \leq T} |\Gamma_t^{(n+1)}(\root) - \Gamma_t^{(n)}(\root)| > \beta^n \right)
\leq - h \log \beta + \alpha_T(h).
$$
As $\alpha_T(h) < 0$, we can pick $\beta < 1$ so that the right hand side is less than zero.  Now Apply Borel-Cantelli to conclude that $\Gamma_t^{(n)}(\root)$ is a Cauchy sequence in $n$, with a rate of convergence that is uniform in $t$. This proves the first part of the theorem.

To prove the martingale property, simply note that by Corollary \ref{corollary:Lp_convergence} there is an $h > 1$ such that $\Gamma_t^{(n)}(\root)$ converges to $\Gamma_t(\root)$ in $L^h$, and hence in $L^1$. Thus for $A \in \mathcal{W}_s$
\begin{align*}
\expect{( \Gamma_{t+s}(\root) - \Gamma_t(\root) ) \textbf{\textrm{1}}_A } = \lim_{n \to \infty} \expect{(\Gamma_{t+s}^{(n)}(\root) - \Gamma_t^{(n)}(\root)) \textbf{\textrm{1}}_A } = 0,
\end{align*}
with the last equality using the martingale property of the finite level measure process. This proves that $\Gamma_t$ is a martingale with respect to $\mathcal{W}_t$, but since $\mathcal{F}_t \subset \mathcal{W}_t$ and $\Gamma_t$ is $\mathcal{F}_t$-measurable, it is automatically a martingale with respect to $\mathcal{F}_t$ also.

For the continuity statement observe that if $W_t$ is continuous then so is $\Gamma_t^{(n)}(\root)$, since it is a finite product and sum of continuous functions. The Borel-Cantelli argument above gives continuity of $\Gamma_t(\root)$ by completeness of $C([0,T])$ under the sup norm.

Finally, if $v \neq \root$ then the proofs above are easily extended by noting that $W_T$-regularity is inherited by the submeasures $\Gamma_{|v}$ (see Remark \ref{remark:submeasure_regularity}). The simple relation $\Gamma_t(v) = X_t(v) \mathcal{C}(\Gamma_{|v}, W_t)(v)$ finishes the argument, and since there are only countably many vertices on the tree the proof is completed.
\end{proof}

\subsection{The Markov Property \label{sec:markov_property}}

In this section we show that the $\Gamma_t$ process has the Markov property. For a given weight process $W_t$ on $[0,T]$, let $\mathcal{M}_T$ be the space of measures $\Gamma$ that satisfy Assumption \ref{assumption:basic}. The Markov property can be formally stated by saying that for any bounded, measurable $F : \mathcal{M}_T \to \R$ we have
\begin{align*}
\condexpect{F(\Gamma_{t+s})}{\mathcal{F}_t} =  \condexpect{F(\Gamma_{t+s})}{\Gamma_t},
\end{align*}
for $s,t \geq 0$ such that $s+t \leq T$. By a density argument it is sufficient to consider the functions of the form $F_v(\Gamma) = \Gamma(v)$ for $v \in \mt$. For these functions we will actually prove the stronger statement
\begin{align*}
\condexpect{F_v(\Gamma_{t+s})}{\mathcal{W}_t} = \condexpect{F_v(\Gamma_{t+s})}{\Gamma_t},
\end{align*}
the difference between the two being that $\mathcal{W}_t$ is a coarser $\sigma$-algebra than $\mathcal{F}_t$. Since the weight processes $s \mapsto W_{t,t+s}$ are independent of $\mathcal{W}_t$, it is sufficient to prove the following:

\begin{theorem} \label{theorem:markov}
Under Assumptions \ref{assumption:basic}, for fixed $s, t \geq 0$ such that $t+s \leq T$, we have with probability one that
\begin{align*}
\Gamma_{t+s} = \mathcal{C} \left( \Gamma_t; W_{t, t+s} \right).
\end{align*}
\end{theorem}

\begin{proof}
We will show that for every $v$ in $\mt$,
\begin{align}
\Gamma_{t+s}(v) = \mathcal C (\Gamma_t; W_{t, t+s})(v). \label{markov_condition_for_v}
\end{align}
We first concentrate on the case $v = \root$. Note that both sides of equation  \eqref{markov_condition_for_v} are defined as limits, and it suffices to prove that
\begin{align}  \label{sufficient_markov_condition_2}
\lim_{n \to \infty } \mathcal C (\Gamma_t; X_{t, t+s})^{(n)} (\root) - \Gamma_{t+s}^{(n)}(\root) = 0.
\end{align}
We will show that the left hand side of \eqref{sufficient_markov_condition_2} goes to zero in $L^h$ for any $h \in (1,h_T)$, and therefore the a.s. limit must be zero as well. Fix $h \in (1, h_T)$ and recall that
\begin{align*}
\mathcal C (\Gamma_t; X_{t, t+s})^{(n)}(\root) = \vsum \Gamma_t(v) X_{t , t+s}(v) = \vsum \frac{\Gamma_t(v)}{X_t(v)}X_{t+s}(v).
\end{align*}
The last equality follows since $X_{t}X_{t,t+s} = X_{t+s}$ by construction. Therefore, by definition of $\Gamma_{t+s}^{(n)}$,
\begin{align*}
\mathcal C (\Gamma_t; X_{t, t+s})^{(n)}(\root) - \Gamma_{t+s}^{(n)}(\root) = \vsum \left( \frac{\Gamma_t(v)}{X_t(v)} - \Gamma(v) \right) X_{t+s}(v).
\end{align*}
Now note that the random variables $\{ \Gamma_t(v)/X_t(v) - \Gamma(v)  : |v| =n \}$ are mean zero, and each depends only on the $W_t$ weights in the subtree $\mt(v)$. Hence they are independent of each other \textit{and} of all the weight processes $t \mapsto W_t(v)$ with $|v| \leq n$. In particular each $X_{t+s}(v)$, for $|v|=n$, is independent of these random variables. Thus we can apply the von Bahr-Esseen inequality to the difference above to get
\begin{align*}
\expect{\left| \mathcal C (\Gamma_t; X_{t, t+s})^{(n)}(\root) - \Gamma_{t+s}^{(n)}(\root) \right|^h}
& \leq \vsum \expect{X_{t+s}(v)^h}  \expect{ \left| \frac{\Gamma_t(v)}{X_t(v)} - \Gamma(v) \right|^h} \\
& = \vsum \expect{X_{t,t+s}(v)^h} \expect{\left| \Gamma_t(v) - X_t(v) \Gamma(v) \right|^h}.
\end{align*}
Recognizing that $X_t(v) \Gamma(v) = \Gamma_t^{(n)}(v)$ and applying Lemma \ref{lemma: cascade_Fan} finishes the proof, since for $h \in (1, h_T)$ we have
\begin{align*}
\limsup_{n \to \infty} \frac{1}{n} \log \expect{\left| \mathcal C (\Gamma_t; X_{t, t+s})^{(n)}(\root) - \Gamma_{t+s}^{(n)}(\root) \right|^h} &\leq \log \expect{W_{t,t+s}^h} + \lambda_{\Gamma}(h) + \log \expect{W_t^h} \\
&= \lambda_{\Gamma}(h) + \log \expect{W_{t+s}^h} \\
&\leq \alpha_T(h) \\
& < 0.
\end{align*}
The second inequality follows from Remark \ref{remark:increasing_moments}, and the last is by the $W_T$-regularity of $\Gamma$.

The proof for $v \neq \root$ is similar, with all sums in the above statements being replaced with sums over the appropriate subtrees, and by making use of the fact that $\Gamma_{|v}$ is $W_T$-regular for all $v$. Finally, since there are only countably many vertices on the tree the statement holds for all vertices simultaneously.
\end{proof}

Note that Theorem \ref{theorem:markov} assumes nothing about the regularity of $\Gamma_t$, even though we applied the cascading procedure to it. Theorem \ref{theorem:markov} gives that $\mathcal{C}(\Gamma_t, W_{t,t+s})$ is indeed a non-trivial measure since it is equal to $\Gamma_{t+s}$, which was already known to be non-trivial by the $W_T$-regularity of the original measure $\Gamma$. However, the regularity condition is only a sufficient one, and so the fact that $\mathcal{C}(\Gamma_t, W_{t,t+s})$ is non-trivial does not imply that $\Gamma_t$ is $W_{t,T}$ regular. This regularity statement is true, however, and we will now prove it. In some sense this gives a classification of the state space of the Markov process: each $\Gamma_t$ lives in the space of $W_{t,T}$-regular measures, which is itself contained in the space of $W_T$-regular measures.

\begin{lemma} \label{lemma:cascade_P}
Under Assumptions \ref{assumption:basic}, the measures $\Gamma_t$ are $W_{t,T}$-regular for each $t \leq T$.
\end{lemma}

\begin{proof}
From Corollary \ref{cascade_Moment}, for $h \in (1, h_T)$, we get that
\begin{align*}
\limsup_{n \to \infty} \frac{1}{n} \log \vsum \expect{\Gamma_t(v)^h}
& \leq \alpha_t(h) < \alpha_T(h) < 0.
\end{align*}
Therefore an application of Borel-Cantelli implies that
\begin{align*}
\lambda_{\Gamma_t}(h) = \limsup_{n \to \infty} \frac{1}{n} \log \vsum \Gamma_t(v)^h \leq \lambda_{\Gamma}(h) + \log \expect{W_t^h}
\end{align*}
with probability one. This gives that
\begin{align*}
\lambda_{\Gamma_t}(h)  + \log \expect{W_{t,T}^h } \leq \lambda_{\Gamma}(h) + \log \expect{W_t^h} + \log \expect{W_{t,T}^h} = \lambda_{\Gamma}(h) + \log \expect{W_T^h}
\end{align*}
for all $h \in (1, h_T)$. Now apply the Mean Value Theorem and take $h \downarrow 1$ to finish the proof.
\end{proof}

\section{Gaussian Weight Processes \label{sec:gaussian}}

The simplest case of weights satisfying Assumption \ref{assumption:basic} is an exponential of a Brownian motion, properly normalized. In this section we study some extra properties of the random cascade process with these weights; specifically we derive stochastic calculus formulas for the evolution of the measures as driven by the Brownian noise. We restrict ourselves to the simplest case when $t \mapsto \log W_t$ has stationary increments, so that
\begin{align*}
W_t(v) = \exp \left \{ B_t(v) - t/2 \right \},
\end{align*}
where $\{ B_t(v) \}_{v \in \mt}$ is a collection of independent Brownian motions with $B_0(v) = 0$. Since the $W_t$ variables have moments of all orders for all $t \geq 0$, we only need to assume that the initial measure $\Gamma$ is $W_T$-regular for some $T > 0$. It is easy to compute that
\begin{align*}
\expect{W_t \log W_t} = - \frac{t}{2},
\end{align*}
so therefore the cascade process is well-defined on $[0, -2 \lambda_{\Gamma}'(1+))$. It is straightforward to verify from the definition of $\lambda_{\Gamma}$ that $-2 \lambda_{\Gamma}'(1+)$ is maximal when $\Gamma = \theta$, and this maximum value is $2 \log 2$. Moreover, for any $ T <  -2 \lambda_{\Gamma}'(1+)$, Assumptions \ref{assumption:basic} are satisfied and hence the $t \mapsto \Gamma_t$ process is always defined on a finite time interval. By Theorems \ref{theorem:existence} and \ref{theorem:markov} the process is Markovian, \textit{and} $t \mapsto \Gamma_t(v)$ is a continuous martingale for each $v \in \mt$. Since
\begin{align*}
X_t(\xi_n) = \exp \left \{ \sum_{i=1}^n B_t(\xi_n) - nt/2 \right \},
\end{align*}
it is easy to compute that
\begin{align*}
\frac{dX_t(\xi_n)}{X_t(\xi_n)} = \sum_{i=1}^n dB_t(\xi_i).
\end{align*}
Therefore
\begin{align}\label{eqn:level_n_SDE}
d \Gamma_t^{(n)}(\root) &= \int_{\partial \mt} X_t(\xi_n) \left( \sum_{i=1}^n dB_t(\xi_n) \right) \, d \Gamma(\xi) = \sum_{i=1}^n \int_{\partial \mt} dB_t(\xi_i) d \Gamma_t^{(n)}(\xi).
\end{align}
This leads to the following result:

\begin{proposition}\label{prop:total_mass_SDE}
The total mass $\Gamma_t(\root)$ evolves according to the stochastic differential equation
\begin{align}\label{eqn:root_SDE}
d \Gamma_t(\root) = \sum_{i=1}^{\infty} \int_{\partial \mt} dB_t(\xi_i) \, d\Gamma_t(\xi) = \sum_{i=1}^{\infty} \E_{\Gamma_t} \left[ dB_t(\xi_i) \right] = \sum_{\substack{v \in \mt \\ v \neq \root}} \Gamma_t(v) \, dB_t(v),
\end{align}
where all stochastic differentials are understood in the It\^{o} sense. Equivalently
\begin{align*}
\frac{d \Gamma_t(\root)}{\Gamma_t(\root)} = \sum_{i=1}^{\infty} \E_{\Gamma_t^*} \left[ dB_t(\xi_i) \right] = \sum_{\substack{v \in \mt \\ v \neq \root}} \Gamma_t^*(v) \, dB_t(v),
\end{align*}
where $\Gamma_t^*$ is $\Gamma_t$ normalized to be a probability measure. The quadratic variation of the latter process is
\begin{align*}
\frac{d \Bigr \langle  \Gamma_t(v), \Gamma_t(v) \Bigl \rangle}{\Gamma_t(v)^2} = \sum_{i=1}^{\infty} \E_{\Gamma_t^* \times \Gamma_t^*} \left[ \indicate{\xi_i = \xi_i'} \right] = \sum_{\substack{v \in \mt \\ v \neq \root}} \Gamma_t^*(v)^2 = \sum_{\substack{v \in \mt \\ v \neq \root}} \left( \frac{\Gamma_t(v)}{\Gamma_t(\root)} \right)^2.
\end{align*}
\end{proposition}

Before proceeding with the proof we first note that all of the stochastic integrals
\begin{align*}
\int_0^t \Gamma_s^{(n)}(v) \, dB_s(v), \quad \int_0^t \Gamma_s(v) \, dB_s(v)
\end{align*}
are well-defined on $[0,T]$. Both integrands are clearly progressively measurable, and as they are continuous local martingales in time their supremum is almost surely finite on the compact interval $[0,T]$. Hence
\begin{align*}
\int_0^T \Gamma_s^{(n)}(v)^2 \, ds < \infty \quad \textrm{and} \quad \int_0^T \Gamma_s(v)^2 \, ds < \infty
\end{align*}
with probability one, which is exactly what is required for the integrals to make sense. Note, however, that the expectations of the latter integrals will not necessarily be finite for all $T$.

\begin{proof}
By the definition of $\Gamma_t(\root)$ as the limit of $\Gamma_t^{(n)}(\root)$, and computing the difference between \eqref{eqn:level_n_SDE} and \eqref{eqn:root_SDE}, it is sufficient to show that the process
\begin{align}\label{eq:SDE_diff}
t \mapsto \sum_{i=1}^n \sum_{|v|=i} \int_0^t \left( \Gamma_s(v) - \Gamma_s^{(n)}(v) \right) \, dB_s(v) + \sum_{i=n+1}^{\infty} \sum_{|v|=i} \int_0^t \Gamma_s(v) \, dB_s(v)
\end{align}
goes to zero in some sense as $n \to \infty$. We will show that the supremum over $[0,T]$ goes to zero almost surely. Our main tool will be the Burkholder-Davis-Gundy inequality, see \cite[Ch. IV, Corollary 4.2]{RV99} for details.

As the quadratic variation of the first summation in \eqref{eq:SDE_diff} is
\begin{align*}
Q_t := \sum_{i=1}^n \sum_{|v|=i} \int_0^t \left( \Gamma_s(v) - \Gamma_s^{(n)}(v) \right)^2 \, ds,
\end{align*}
the BDG inequality gives us that for $h > 0$ there is a constant $C_h > 0$ such that
\begin{align*}
\expect{\sup_{0 \leq t \leq T} \left| \sum_{i=1}^n \sum_{|v|=i} \int_0^t \left( \Gamma_s(v) - \Gamma_s^{(n)}(v) \right) \, dB_s(v) \right|^h} & \leq C_h \expect{Q_t^{h/2}}.
\end{align*}
Choose $h \leq 2$ so that, by subadditivity and a supremum bound on the integral terms, the right hand side is bounded above by
\begin{align*}
C_h T^{h/2} \sum_{i=1}^n \sum_{|v|=i} \expect{\sup_{0 \leq t \leq T} |\Gamma_t(v) - \Gamma_t^{(n)}(v)|^h}.
\end{align*}
Now by choosing $h > 1$, Doob's maximal inequality gives that this is further bounded above by
\begin{align*}
C_h^* T^{h/2} \sum_{i=1}^n \sum_{|v|=i} \expect{|\Gamma_T(v) - \Gamma_T^{(n)}(v)|^h}.
\end{align*}
By Lemma \ref{lemma: cascade_Fan} the latter term goes to zero exponentially fast as $n \to \infty$, and then Borel-Cantelli completes the proof.

For the second summation of \eqref{eq:SDE_diff}, the same argument with the BDG inequality yields that
\begin{align}\label{eq:SDE_second}
\expect{\sup_{0 \leq t \leq T} \left| \sum_{i=n+1}^{\infty} \sum_{|v|=i} \int_0^t \Gamma_s(v) \, dB_s(v) \right|^h } \leq C_h^* T^{h/2} \sum_{i=n+1}^{\infty} \sum_{|v|=i} \expect{\Gamma_T(v)^h}.
\end{align}
From the proof of Lemma \ref{lemma:cascade_P} we have that
\begin{align*}
\limsup_{n \to \infty} \frac{1}{n} \log \sum_{|v|=n} \expect{\Gamma_T(v)^h} \leq \lambda_{\Gamma}(h) + \log \expect{W_T^h} < 0
\end{align*}
for $h$ sufficiently close to $1$, and hence by \eqref{eq:SDE_second} and the Borel-Cantelli lemma the second summation of \eqref{eq:SDE_diff} term goes to zero almost surely.
\end{proof}

Using equation \eqref{eq:subtree_mass} this leads to the following formulas for the evolution at other vertices:

\begin{corollary}\label{corollary:total_mass_SDE}
For $v \in \mt$ the mass $\Gamma_t(v)$ evolves as
\begin{align*}
\frac{d \Gamma_t(v)}{\Gamma_t(v)} = \sum_{i=1}^n dB_t(v_i) + \sum_{\substack{u \in \mt(v) \\ u \neq v}} \frac{\Gamma_t(u)}{\Gamma_t(v)} \, dB_t(u),
\end{align*}
where $\root = v_0, v_1, v_2, \ldots, v_n = v$ are the vertices from the root to $v$. In particular this gives that if $u$ is not a descendant of $v$ or vice-versa then
\begin{align*}
\frac{d \,\Bigl \langle \Gamma_t(u), \Gamma_t(v) \Bigr \rangle}{ \Gamma_t(u) \Gamma_t(v)} = |u \wedge v| \, dt,
\end{align*}
where $u \wedge v$ is the last common ancestor of the paths to $u$ and $v$.
\end{corollary}

Proposition \ref{prop:total_mass_SDE} says that the total mass evolves as a continuous time exponential martingale. Its logarithm accumulates quadratic variation at a rate given by the last expression of Proposition \ref{prop:total_mass_SDE}, and, as is well known, the time at which an exponential martingale hits zero is equivalent to the time at which the accumulated quadratic variation reaches infinity. This gives another interpretation of the lifetime of the $\Gamma_t$ process:

\begin{corollary}\label{corollary:explosion_time}
The $\Gamma_t$ process reaches the zero measure at exactly the time
\begin{align*}
\sup \left \{ t \geq 0 : \int_0^t \sum_{\substack{v \in \mt \\ v \neq \root}} \left( \frac{\Gamma_s(v)}{\Gamma_s(\root)} \right)^2 ds < \infty \right \}.
\end{align*}
\end{corollary}

Before this time, that the total mass process is an exponential martingale naturally suggests the Girsanov theory plays a role here. This leads to the following:

\begin{corollary}\label{corollary:girsanov}
Let $P$ be the measure under which the vertex processes $\{ B_t(v) \}_{v \in \mt}$ are independent Brownian motions. Assume that $\Gamma$ is a probability measure. For any $T' < T$, let $\tilde{P}_{T'}$ be the probability measure whose Radon-Nikodym derivative with respect to $P$ is $\Gamma_{T'}(\root)$. Then under $\tilde{P}_{T'}$ the processes
\begin{align*}
\left \{ t \mapsto B_t(v) - \int_0^t \Gamma_s(v) \, ds, 0 \leq t \leq T' \right \}_{v \in T}
\end{align*}
are independent Brownian motions on the tree vertices.
\end{corollary}

See \cite{RV99} for background on the Girsanov theory. In Section \ref{sec:applications} we describe an application of this result to the model of tree polymers.

\section{H\"{o}lder Continuity \label{sec:Holder}}

This section highlights an interesting application of the SDE results derived in the previous section. Once again we will assume that the weight processes are
\begin{align*}
W_t(v) = \exp \left \{ B_t(v) - t/2 \right \},
\end{align*}
where $\{ B_t(v) \}_{v \in \mt}$ is a collection of independent Brownian motions with $B_0(v) = 0$. Recall from Section \ref{sec:gaussian} that $\Gamma_t$ is a well defined measure-valued process on the time interval $[0, -2 \lambda_{\Gamma}'(1+))$. Using techniques from stochastic analysis we will show that this process $\Gamma_t$ is $\alpha$-H\"{o}lder in the Wasserstein metric for any $\alpha < 1/2$. This gives an interesting juxtaposition of discontinuity and continuity. On the one hand, the measures $\Gamma_t$ and $\Gamma_s$ are mutually singular for $t \neq s$, and hence are very discontinuous in the total variation distance. However at the same time, they  satisfy a very strong continuity condition in the Wasserstein metric on probability measures on the tree.

\begin{definition}
The Wasserstein distance between any two probability measures $\mu$ and $\nu$ on $\partial \mt$ is defined as
\begin{align*}
d_W(\mu, \nu) := \inf_{\rho \in \Lambda(\mu, \nu) } \int_{\partial \mt \times  \partial \mt}  d(\zeta, \eta) \text{d}\rho(\zeta, \eta),
\end{align*}
where $\Lambda(\mu, \nu)$ is the collection of all couplings of the measures $\mu$, $\nu$. Recall that the distance function is $d(\zeta, \eta) = 2^{-|\zeta \wedge \eta|}$.
\end{definition}

Note that this is a metric on probability measures on $\partial \mt$. Our main result in this section applies to the normalized process $\tilde{\Gamma}_t := \Gamma_t/\Gamma_t(\root)$.

\begin{theorem} \label{w_Holder}
Let $T <  -2 \lambda_{\Gamma}'(1+)$. Then for any $\alpha < 1/2$, the process $\tilde{\Gamma}_t$, $0 \leq t \leq T$, is $\alpha$-H\"{o}lder continuous in the Wasserstein metric.
\end{theorem}

The main step in the proof is to show H\"{o}lder continuity of each of the processes $\Gamma_t(v)$, for $v \in \mt$, along with a bound on the H\"{o}lder constant.

\begin{theorem}\label{Kolmogorov_Cascade}
Let $T < -2 \lambda_{\Gamma}'(1+)$. Then for any $v \in \mt$, the processes $\Gamma_t(v)$ are $\alpha$-H\"{o}lder continuous on $[0,T]$ for any $\alpha < 1/2$. Moreover, there is a $\gamma < 1$ such that,
\begin{align*}
\sup_{v \in \mt} \sup_{0 \leq s < t \leq T  }  \gamma^{-|v|} \frac{|\Gamma_{t}(v) - \Gamma_{s}(v)|}{ |t-s|^{\alpha}} < \infty
\end{align*}
almost surely.
\end{theorem}

We use the following version of the Kolmogorov-Chentsov Theorem, which gives a bound on the magnitude of the H\"{o}lder constant. For a statement of this theorem see \cite[Theorem 2.2.8]{KS91}. The statement on the control of the H\"{o}lder constant is implicit in their proof.

\begin{theorem} [Kolmogorov-Chentsov Theorem] \label{Kolmogorov}
Let $X_t$ be a continuous, stochastic process on $[0,T]$ such that for all $t, s \leq T$,
\begin{align*}
\E \left| X_{t} - X_s\right|^{p} < K_p |t-s|^{p/2}
\end{align*}
for some $p>2$ and some constant $K_p$. Then $X_t$ is $\alpha$-H\"{o}lder continuous  for every $\alpha <1/2 -  1/p$. Moreover,
\begin{align}
\P \left(\sup_{0 \leq s < t \leq T  }  \frac{|X_{t} - X_{s}|}{|t-s|^{\alpha}}> 1 \right) \leq K_p H_{\alpha}, \label{p_H_bound}
\end{align}
where $H_{\alpha}$ is a constant depending only on $\alpha$ and $T$.
\end{theorem}

To prove Theorem \ref{Kolmogorov_Cascade} we restrict the process to a sequence of stopping times, prove H\"{o}lder continuity of these stopped process, and then take a limit. We construct these stopped processes in the following lemma. Note that in this section, and this section only, the notation $\Gamma_t^{(N)}$ refers to the stopped version of the $\Gamma_t$ process, not to the finite level cascade measure $\Gamma_t^{(n)}$ as in all other sections.

\begin{lemma} \label{stop_Holder}
Let $T <  -2 \lambda_{\Gamma}'(1+)$. Then there is a $\gamma < 1$ and a sequence of measure-valued processes $\Gamma_t^{(N)}$ for $N \in \mathbb{N}$, such that
\begin{enumerate}[(i)]
\item $\P(\Gamma_t^{(N)} \neq \Gamma_t \text{ for some } t \leq T)  \to 0$ as $N \to \infty$,
\item for every $v \in \mt$, $\Gamma_t^{(N)}(v)$ is $\alpha$-H\"{o}lder on $[0,T]$ for any $\alpha < 1/2$,
\item for  $\alpha< 1/2$, we have with probability one that,
\begin{align*}
\sup_{v \in \mt}   \sup_{0 \leq s < t \leq T} \gamma^{-|v|} \frac{|\Gamma_{t}^{(N)}(v) - \Gamma_{s}^{(N)}(v)|}{|t-s|^{\alpha}} < \infty.
\end{align*}
\end{enumerate}
\end{lemma}

\begin{remark}
Theorem \ref{Kolmogorov_Cascade} follows immediately from this lemma.
\end{remark}

The processes $\Gamma_t^{(N)}$ will be $\Gamma_t$ stopped at an appropriate stopping time. We construct these stopping times in the following lemma.

\begin{lemma} \label{stopping_Time}
For any $T<   -2 \lambda_{\Gamma}'(1+)$ there is a $\beta < 1$ and a sequence of stopping times $\tau_N$ with $\P(\tau_N < T) \to 0$ as $N \to \infty$ such that,
\[
\sup_{v \in \mt} \sup_{0 \leq t \leq T} \beta^{-|v|} \Gamma_{t \wedge \tau_N} (v) \leq C N,
\]
for some constant $C$ depending on $\Gamma$ and $\beta$.
\end{lemma}

\begin{proof}
Fix $T<  -2 \lambda_{\Gamma}'(1+)$. Now recalling Definition \ref{h_time} and Remark \ref{remark:h_time}, we take $h \in (1, h_T)$ and note that $\alpha_T(h) < 0$. We can therefore choose $\beta$ such that $\alpha_T(h)/h < \log \beta < 0$; hence $\beta < 1$. Consider the continuous, increasing processes
\begin{align*}
A_t(v) := {\beta^{-|v|}} \sup_{0 \leq s \leq t } \Gamma_s(v).
\end{align*}
It follows from the continuity of $\Gamma_t(v)$ that $A_{t}(v)$ is bounded on $[0,T]$ for every $v \in \mt$. Now define
\begin{align}\label{defn:At}
A_t := \sup_{v \in \mt} A_t(v).
\end{align}
It follows  from our choice of $\beta$ and the definition of $\alpha_T(h)$ that $A_0$ is non-random and finite. Clearly $A_t$ is a non-decreasing process. Note that the statement of the lemma is equivalent to finding a sequence of stopping times $\tau_N$ such that $A_{t \wedge \tau_N} \leq A_0 + N$ and with $\P(\tau_N < T) \to 0$ as $N \to \infty$. This will follow from the fact that $A_t$ is continuous on $[0,T]$, which we now show. Using Markov's inequality as well as Doob's $L^p$ inequality we get that
\begin{align*}
 \P \left( A_T(v) \geq A_0 \text{ for some } |v| = n \right)
& \leq A_0^{-h} \beta^{-nh} \sum_{|v| = n}  \E \left[  \Gamma_T(v)^h \right].
\end{align*}
By  Corollary \ref{cascade_Moment} and the choice of $\beta$,
\begin{align*}
\limsup_{n \to \infty} \frac{1}{n}  \log  \P \left( A_T(v) \geq A_0 \text{ for some } |v| =n \right)
\leq -h \log \beta +  \alpha_T(h)
 < 0.
\end{align*}
Therefore by Borel-Cantelli,
\begin{align*}
\P \left( \sup_{0 \leq t \leq T} A_t(v) \geq A_0 \text{  for only finitely many } v \in \mt  \right) = 1.
\end{align*}
Take $S = \{ v \in \mt : A_{T}(v) \geq A_0 \}$ to be the (random) set of vertices from the tree at which this inequality fails; clearly $S$ is finite. Moreover, since the processes $A_t(v)$ and $A_t$ are non-decreasing in $t$, it follows that the supremum in \eqref{defn:At} can only be achieved at one of the vertices of $S$, i.e.
\begin{align*}
A_t = \max_{v \in S} A_t(v).
\end{align*}
Hence $A_t$ is itself continuous on $[0,T]$, since it is the maximum of a finite number of continuous processes. We then take
\[
\tau_{N} : = \inf \left \{ t \in [0,T] \, : A_t > A_0 + N \right \}
\]
to be our sequence of stopping times. The continuity of $A_t$ implies that $A_{t \wedge \tau_n} \leq A_0 + N$ as well as the fact that $A_t$ is almost surely bounded on $[0,T]$. The boundedness on $[0,T]$ also gives that $\P(\tau_N < T) \to 0$ as $N \to \infty$.
\end{proof}

Proving Lemma \ref{stop_Holder} now becomes an application of the Kolmogorov-Chentsov Theorem.

\begin{proof}[Proof of Lemma \ref{stop_Holder}]
First take  $\beta < 1$ and $\tau_N$ as in Lemma \ref{stopping_Time} and define $\Gamma_t^{(N)} := \Gamma_{t \wedge \tau_N}$ to be the stopped version of the measure-valued process. From  Lemma \ref{stopping_Time}, part (i) of this lemma is immediate.

Next, recall that by Corollary \ref{corollary:total_mass_SDE}
\begin{align*}
d\Gamma_t(v) = \sum_{i=1}^{|v|} \Gamma_t(v) \, dB_t(v_i) + \sum_{\substack{u \in \mt(v) \\ u \neq v}} \Gamma_t(u) \, dB_t(u).
\end{align*}
Now fix $v \in \mt$, $p>2$  and take any $0 \leq s < t \leq T$.  We apply the Burkholder-Davis-Gundy inequality, use the bound on the process $\Gamma_t^{(N)}$ from Lemma \ref{stopping_Time}, and the fact that it is a flow on $\mt$ to get
\begin{align*}
\E \left| \Gamma_t^{(N)}(v) - \Gamma_s^{(N)}(v)  \right|^p
&\leq \E \left( \sum_{i=1}^{|v|} \int_{s \wedge \tau_n}^{t \wedge \tau_n } \Gamma_{r}(v)^2 \, dr  + \sum_{\substack{u \in \mt(v) \\ u \neq v }} \int_{s \wedge \tau_N}^{t \wedge \tau_n } \Gamma_r(v)^2 \, dr \right)^{p/2}\\
& \leq \E \left( \sum_{i=1}^{|v|}C  N \beta^{|v|} \int_{s \wedge \tau_n}^{t \wedge \tau_n } \Gamma_{r}(v) \, dr  + \sum_{k=1}^{\infty}\sum_{ \substack{u \in \mt(v) \\ u = |v| +k }}  N \beta^{|u|} \int_{s \wedge \tau_N}^{t \wedge \tau_n } \Gamma_r(u) \, dr \right)^{p/2}\\
&= \E \left( |v| C N \beta^{|v|} \int_{s \wedge \tau_n}^{t \wedge \tau_n }\Gamma_{r}(v) \, dr + \sum_{k=1}^{\infty} N \beta^{|v|+k} \int_{s \wedge \tau_N}^{t \wedge \tau_n } \Gamma_r(v) \, dr    \right)^{p/2}.
\end{align*}
Now again use the upper bound $\Gamma^{(N)}_t(v) \leq C N \beta^{|v|}$ and the fact that $\beta < 1$ to get the desired Kolmogorov-Chentsov inequality,
\begin{align}
\E \left| \Gamma^{(N)}_{t}(v) - \Gamma_s^{(N)}(v)  \right|^p
& \leq \E \left( C^2 N^2 \beta^{2|v|} (|v| + C_{\beta}) \int=_{s \wedge \tau_N}^{t \wedge \tau_N} \text{d}s \right)^{p/2}  \notag \\
& \leq C^p N^p \beta^{p|v|}(|v| + C_{\beta})^{p/2} (t-s)^{p/2} \label{Lp_KC},
\end{align}
where $C_\beta$ is a constant depending only on $\beta$. Since this inequality holds for every $p>2$,we get that for every $v \in \mt$, $\Gamma_t^{(N)}(v)$ is $\alpha$-H\"{o}lder continuous for any $\alpha < 1/2$.

Finally, fix $\alpha < 1/2$ and take $\gamma \in (\beta, 1)$. The $L^p$ bound \eqref{Lp_KC}, applied to the process $\gamma^{-|v|} \Gamma_t^{(N)}(v)$, along with the Kolmogorov-Chentsov bound \eqref{p_H_bound} gives that
\begin{align}
\P\left(\sup_{0 \leq s < t \leq T} \gamma^{-|v|} \frac{| \Gamma_{t}^{(N)}(v) - \Gamma_s^{(N)}(v)|}{ |t-s|^{\alpha}}  > 1 \right) < K_p^N(v) H_{\alpha}, \label{ptH_bound}
\end{align}
where $K_p^N(v) = C^p N^p (\beta/\gamma) ^{p |v|} (|v| + C_{\beta})^{p/2}$. Notice that $K_p^N(v) \to 0$ as $p \to \infty$, at least for all $|v| > M$ where $M > 0$ depends on only $\beta$, $\gamma$, and $N$. Therefore, since \eqref{ptH_bound} is true for every $p > 2$, it follows that for all $|v| > M$ we have
\begin{align*}
\P \left( \sup_{0 \leq s < t \leq T} \gamma^{-|v|} \frac{| \Gamma_{t}^{(N)}(v) - \Gamma_s^{(N)}(v)|}{ |t-s|^{\alpha}} > 1 \right) = 0,
\end{align*}
which implies part (iii) of the lemma.
\end{proof}

The last ingredient in the proof of Theorem \ref{w_Holder} is the following upper bound on the Wasserstein distance on $M(\mt)$.

\begin{lemma} \label{w_UpBound}
Let $\mu$, $\nu$ $\in M(\mt)$ be such that for every $v \in \mt$ we have $\mu(v), \nu(v) > 0$. Then
\begin{align*}
d_W(\mu, \nu) \leq \sum_{k=1}2^{-k+1} \sum_{|v|=k-1} \nu(v) \left|   \frac{\nu(v_L)}{\nu(v)} -  \frac{\mu(v_L)}{\mu(v)}\right|.\
\end{align*}
\end{lemma}

\begin{proof}
This follows from a particular, standard coupling $\rho$ of the two measures $\mu$ and $\nu$. Given a ray $\xi \in \partial \mt$, we define the probability measure
$\nu_{\xi}$ on $\partial \mt$ via the following iterative formula:
\begin{align*}
\nu_{\xi} (\eta_k | \eta_{k-1}) :=
\begin{cases}
\frac{\nu(\eta_k)}{\nu(\eta_{k-1})} & \text{ if } \eta_{k-1} \neq \xi_{k-1} \\
p_k(\xi) & \text{ if } \eta_k = \xi_k \\
1- p_k(\xi) & \text{ if }  \eta_{k-1} = \xi_{k-1} \text{ but }  \eta_k \neq \xi_k
\end{cases}
\end{align*}
where
\begin{align*}
p_k(\xi) =  \left( \frac{\nu(\xi_k)}{\nu(\xi_{k-1})} \frac{\mu(\xi_{k-1})}{\mu(\xi_{k})} \right) \wedge 1.
\end{align*}
We define the coupling $\text{d} \rho(\xi, \eta) = \text{d}\mu(\xi) \text{d}\nu_{\xi}(\eta)$. In words the coupling is the following. We first sample a ray $\xi$ from $\mu$. Then conditioned on $\xi$ we sample $\eta$ inductively. If $\eta$ agrees with $\xi$ on the first $k-1$ steps of the path (i.e., $\eta_k = \xi_k$), then flip a $p_k(\xi)$ coin to decide if $\eta$ will agree with $\xi$ on the $k$ step. Once $\eta$ diverges from $\xi$, pick the rest of its path independently from $\nu$.

It is a matter of simple calculation to show that this is a coupling. Since $\nu_{\xi}$ is clearly a probability measure on $\partial \mt$, it follows that the first marginal is $\mu$. To compute that the second marginal is $\nu$ is straightforward.

Let $ A_k(\xi) = \{ \eta \in \partial \mt :  \eta_{k-1} =  \xi_{k-1}, \, \eta_k \neq \xi_k  \}$ be the event that $\eta$ agrees with $\xi$ exactly up to level $k$. Hence $d(\xi, \eta) = 2^{-k}$ for $\eta \in A_k(\xi)$. Furthermore
\begin{align*}
\nu_{\xi}(A_k(\xi))
&= \prod_{i=1}^{k-1} p_i(\xi) \cdot (1-p_k(\xi)) \\
& \leq  \prod_{i=1}^{k-1}  \frac{\nu(\xi_i)}{\nu(\xi_{i-1})} \frac{\mu(\xi_{i-1})}{\mu(\xi_{i})} \cdot \left| 1 -  \frac{\nu(\xi_k)}{\nu(\xi_{k-1})} \frac{\mu(\xi_{k-1})}{\mu(\xi_{k})}  \right| \\
& = \frac{\nu(\xi_{k-1})}{\mu(\xi_{k-1})}  \left| 1 -  \frac{\nu(\xi_k)}{\nu(\xi_{k-1})} \frac{\mu(\xi_{k-1})}{\mu(\xi_{k})}  \right| \\
& = \frac{\nu(\xi_{k-1} ) } {\mu(\xi_k)} \left|   \frac{\nu(\xi_{k})}{\nu(\xi_{k-1})} -  \frac{\mu(\xi_{k})}{\mu(\xi_{k-1})}\right|.
\end{align*}
The first equality follows from the definition of $\nu_{\xi}$ while the first inequality follows from the definition of $p_i(\xi)$.
A calculation now gives that
\begin{align*}
d_W(\mu, \nu)
& \leq  \int_{\partial \mt \times  \partial \mt} d(\xi, \eta) \text{d}\mu(\xi) \text{d}\nu_{\xi}(\eta) \\
& = \sum_{k=1}^{\infty} \int_{\partial \mt}\int_ {A_k(\xi)} d(\xi, \eta) \text{d}\mu(\xi) \text{d}\nu_{\xi}(\eta)\\
& = \sum_{k=1}  \int_{\partial \mt} 2^{-k} \text{d}\mu(\xi) \nu_{\xi}(A_k(\xi)) \\
& \leq \sum_{k=1 } 2^{-k} \int_{\partial \mt} \frac{\nu(\xi_{k-1} ) } {\mu(\xi_k)} \left|   \frac{\nu(\xi_{k})}{\nu(\xi_{k-1})} -  \frac{\mu(\xi_{k})}{\mu(\xi_{k-1})}\right| \text{d}\mu(\xi)  \\
&=  \sum_{k=1 } 2^{-k} \sum_{|v| = k } \nu(v_p)  \left|   \frac{\nu(v_k)}{\nu(v_p)} -  \frac{\mu(v_k)}{\mu(v_p)}\right|.
\end{align*}
Recall that $v_p$ denotes the parent of $v$ in $\mathcal{T}$. Finally, noting that $ \left|   \frac{\nu(v_L)}{\nu(v)} -  \frac{\mu(v_L)}{\mu(v)}\right| = \left|   \frac{\nu(v_R)}{\nu(v)} -  \frac{\mu(v_R)}{\mu(v)}\right|$, gives that
\[
\sum_{|v| = k } \nu(v_p)  \left|   \frac{\nu(v_k)}{\nu(v_p)} -  \frac{\mu(v_k)}{\mu(v_p)}\right| = 2 \sum_{|v|=k-1} \nu(v) \left|   \frac{\nu(v_L)}{\nu(v)} -  \frac{\mu(v_L)}{\mu(v)}\right|
\]
which completes the proof.
\end{proof}

We are finally ready to prove Theorem \ref{w_Holder}. It follows from Theorem \ref{Kolmogorov_Cascade} and Lemma \ref{w_UpBound}.

\begin{proof}[Proof of Theorem \ref{w_Holder}]
Let $0 \leq s < t \leq T$ and fix $\alpha < 1/2$. Applying  Lemma \ref{w_UpBound} to the measures $\tilde{\Gamma}_{t}$ and $\tilde{\Gamma}_s$ gives
\begin{align*}
d_W(\tilde{\Gamma}_{s}, \tilde{\Gamma}_t)
& \leq \sum_{k=1} 2^{-k+1}  \sum_{|v|=k-1} \frac{\Gamma_{s}(v)}{\Gamma_s(\root)} \left|   \frac{\Gamma_{s}(v_L)}{\Gamma_{s}(v)} -  \frac{\Gamma_{t}(v_L)}{\Gamma_{t}(v)}\right| \\
& = \frac{1}{\Gamma_s(\root)}\sum_{k=1} 2^{-k+1}  \sum_{|v|=k-1}   \left|   \frac{\Gamma_{s}(v_L) \Gamma_t(v_R)-   \Gamma_{s}(v_R) \Gamma_{t}(v_L)}{\Gamma_{t}(v)}\right|,
\end{align*}
where we have used the fact that $\Gamma_t(v) = \Gamma_t(v_L) + \Gamma_t(v_R)$ for every $t$.  Now note that by Theorem \ref{Kolmogorov_Cascade}, for every $v \in \mt$, $|\Gamma_t(v) - \Gamma_s(v)| \leq C_{\alpha} \gamma^{|v|} |t-s|^{\alpha} $ for some $\gamma < 1$. Therefore, adding and subtracting $\Gamma_t(v_R) \Gamma_t(v_L)$ gives
\begin{align*}
|\Gamma_{s}(v_L) \Gamma_t(v_R)-   \Gamma_{s}(v_R) \Gamma_{t}(v_L)|
& \leq C_{\alpha} \gamma^{|v| + 1}  \Gamma_t(v) |t-s|^{\alpha}.
\end{align*}
This inequality along with the fact that $\Gamma_s(\root)$ is bounded away from zero on $[0,T]$ leads to the $\alpha$-H\"{o}lder inequality in the Wasserstein metric,
\begin{align*}
d_W(\tilde{\Gamma}_{s}, \tilde{\Gamma}_t) &  \leq \frac{1}{\Gamma_s(\root)} \sum_{k=1} 2^{-k+1}  \sum_{|v|=k-1}   C_{\alpha} \gamma^{|v| + 1} |t-s|^{\alpha} \\
& = C_{\alpha}^{'}|t-s|^{\alpha}.
\end{align*}
\end{proof}

%%% Upper Bound of Holder exponent.
This result is optimal in the sense that $\tilde{\Gamma}_t$ is not $\alpha$-H\"{o}lder for any $\alpha > 1/2$ in the Wasserstein metric. This upper bound on the H\"{o}lder exponent follows from general arguments for martingales, which we now briefly outline.

\begin{theorem} \label{not_W_Holder}
For any interval $[a,b] \subset [0,-2 \lambda_{\Gamma}'(1+))$ and for any $\alpha >1/2$,
$$ \limsup_{\epsilon \to 0} \sup_{\substack{a \leq s \leq t \leq b \\ |t-s| \leq \epsilon}} \frac{d_W(\tilde{\Gamma}_t, \tilde{\Gamma}_s) }{|t-s|^{\alpha} } = \infty. $$
\end{theorem}

\begin{proof}
Define
\[
f(\xi) =
\begin{cases}
1 \text{ if } \xi_1 = \root_L \\
0 \text{ if } \xi_1 = \root_R.
\end{cases}
\]
Since $|f(\xi) -  f(\eta)| \leq d(\xi, \eta)$ for any two rays $\xi$, $\eta$ $\in \partial \mt$, we have, using Jensen's inequality, that for any coupling $\rho$ of two probability measures, $\mu$ and $\nu$,
\begin{align*}
 \int_{\partial \mt \times \partial \mt} d(\xi, \eta) \text{d} \rho (\xi, \eta)
&\geq \int_{\partial \mt \times \partial \mt} \left| f(\xi) - f(\eta) \right| \text{d} \rho (\xi, \eta) \\
& \geq \left|  \int_{\partial \mt \times \partial \mt} f(\xi) \text{d} \rho (\xi, \eta) -  \int_{\partial \mt \times \partial \mt} f(\eta) \text{d} \rho (\xi, \eta) \right| \\
& = |\mu(\root_L) - \nu(\root_L)|.
\end{align*}
In particular, this implies that
\[
d_W(\tilde{\Gamma}_t, \tilde{\Gamma}_s) \geq \left| \tilde{\Gamma}_t(\root_L)- \tilde{\Gamma}_s(\root_L) \right|.
\]
So it remains to show that $\tilde{\Gamma}_t(\root_L)$ is not $\alpha$-H\"older for any $\alpha > 1/2$. Both $\Gamma_t(\root_L)$ and $\Gamma_t(\root)$ are non-zero and continuous on $[a,b]$  and so by Ito's formula $\tilde{\Gamma}_t(\root_L)$ is a continuous semi-martingale.  We will use the fact that a continuous semi-martingale whose quadratic variation is strictly increasing is not $\alpha$-H\"{o}lder for any $\alpha > 1/2$ (see Lemma \ref{not_Holder}). For now we only verify that the quadratic variation is strictly increasing. Since
\begin{align*}
\tilde{\Gamma}_t(\root_L) = \frac{\Gtrl}{\Gtrl + \Gtrr},
\end{align*}
by Ito's formula the martingale part of $d \tilde{\Gamma}_t$ is
\[
\frac{1}{\Gt(\root)^2}\Bigl( \Gtrr d \Gtrl - \Gtrl d \Gtrr \Bigr).
\]
From Proposition \ref{prop:total_mass_SDE} it follows that $d \, \bigl \langle \tilde{\Gamma}_t(\root_L) \bigr \rangle \neq 0$.
\end{proof}

For the sake of completeness we provide the following general fact from stochastic calculus that was used in the proof of Theorem \ref{not_W_Holder}.
\begin{lemma} \label{not_Holder}
Let $X_t$ be a continuous semi-martingale, i.e. $X = X_0 + M + A$ where $M$ is a continuous local martingale, $A$ a finite variation process, and $M_0 = A_0 = 0$.  If the quadratic variation $\langle X \rangle _t $ is strictly increasing on some interval $[a,b]$, then for any $\alpha > 1/2$ the process $X_t$ is not $\alpha$-H\"{o}lder continuous on $[a,b]$.
\end{lemma}

\begin{proof}[Proof of Lemma \ref{not_Holder}]
Note that for any $\beta < 1$ and $t \in (a,b)$,
\[
\limsup_{s \to t} \frac{\left|  A_t - A_s \right|}{|t-s|^{\beta}} = 0.
\]
So without loss of generality we can assume that $A = 0$ and that $X$ is a local martingale. First consider the case where there exists a non-random $\delta>0$ such that $\langle X \rangle_b - \langle X \rangle_a > \delta$, with probability one. Fix $\alpha > 1/2$. We will show that $X_t$ is not $\alpha$-H\"{o}lder continuous. For $n \in \mathbb{N}$ and $1 \leq i \leq n$, define the stopping times
$$\tau_{i}^n = \inf_{t} \left \{ \langle X \rangle_t - \langle X \rangle_a > \frac{i}{n} \delta \right \}.$$
Note that $\tau_n^n < b$ and so by the Dubins-Schwarz Theorem,
$$ \left( X_{\tau_{i+1}^n} - X_{\tau_{i}^n} \text{, } i=1,...,n \right) \buildrel d \over = \left( B \left( \frac{i+1}{n} \delta \right) -  B \left( \frac{i}{n} \delta \right) \text{, } i=1,...., n \right), $$
where $B(t)$ is a standard Brownian motion.
Since
$\E \left|  B \left( \frac{i+1}{n} \delta \right) -  B \left( \frac{i}{n} \delta \right) \right|^{\frac{1}{\alpha}}   = C_p \left(  \frac{\delta}{2} \right)^{\frac{2}{\alpha} } n^{-2 \alpha}$,  the weak law of large numbers gives that
\begin{align*}
\sum_{i=1}^n \left|  B \left( \frac{i+1}{n} \delta \right) -  B \left( \frac{i}{n} \delta \right) \right|^{\frac{1}{\alpha}}   \to \infty,
\end{align*}
in probability. Let
\begin{align*}
A_n = \bigcup_{i=1}^n \left \{ |X_{\tau_{i+1}^n} - X_{\tau_{i}^n}| > \left( \tau_{i+1}^n - \tau_{i}^n \right)^{\alpha} \right \}
\end{align*}
be the event that $X_t$ is not $\alpha$-H\"{o}lder at level $n$. By the pigeonhole principle,
\begin{align*}
\P \left(A_n \right) \geq  \P \left(   \sum_{i=1}^n  |X_{\tau_{i+1}^n} - X_{\tau_{i}^n} |^{\frac{1}{\alpha}} > (b-a)   \right).
\end{align*}
The convergence of the right hand side to $1$ gives that
$\P \left( A_n \text{ i.o.} \right) =1. $
Finally, since $\langle X \rangle_t$ is strictly increasing we have that with probability one $\sup_{1 \leq i \leq n}  ( \tau_{i+1}^n - \tau_{i}^n) \to 0$ as $n \to \infty$. This finishes the proof of this case.

Now consider the general case. For every $\epsilon >0$, we can find a $\delta >0$ and non-random $a<T<b$ such that $\P(\langle X \rangle_T - \langle X \rangle_a > \delta) \geq 1-\epsilon$. Consider the process
\begin{align*}
\begin{matrix}
\tilde{M}_t =
\begin{cases}
X_t \qquad t \leq T \\
\begin{cases}
X_t  & \mbox{  } t > T \text{, } \langle X \rangle_T - \langle X \rangle_a \geq \delta \\
X_T + B_{\frac{\delta}{b-T}(t-T)} &  \text{ } t > T  \text{, } \langle X \rangle_T - \langle X \rangle_a < \delta
\end{cases}
\end{cases}
\end{matrix}
\end{align*}
Then $\tilde{X}_t$ is a continuous martingale with $[\tilde{X}]_b - [\tilde{X}]_a > \delta$. Therefore $\tilde{X}_t$ is not $\alpha$-H\"{o}lder for any $\alpha > 1/2$ and so with probability greater than $1- \epsilon$ neither is $X_t$. Since this is true for any $\epsilon$ this completes the proof.
\end{proof}

\section{Applications to Other Models \label{sec:applications}}

\subsection{Tree Polymers \label{sec:polymers}}

Although this paper was written in the language of multiplicative cascades it was strongly motivated by the literature on tree polymers. The polymer model is virtually identical but the language is mildly different: to the vertices of the tree attach iid random variables $\{ \omega(v) \}_{v \in \mt}$, and at inverse temperature $\beta$ and level $n$ define the polymer measure on $\partial \mt$ by
\begin{align*}
d \Gamma_{\omega, \beta}^{(n)}(\xi) := \frac{1}{Z_{\omega, \beta}^{(n)}} \prod_{i=1}^n \exp \left \{ \beta \omega(\xi_i) \right \} \, d\Gamma(\xi).
\end{align*}
Here $Z_{\omega,\beta}^{(n)}$ is the partition function
\begin{align*}
Z_{\omega, \beta}^{(n)} = \int_{\partial \mt} d \Gamma_{\omega, \beta}^{(n)}(\xi) = \Gamma_{\omega, \beta}^{(n)}(\root).
\end{align*}
In the tree polymer model we usually assume that $\Gamma$ is a probability measure, and hence the partition function normalizes the polymer measure to also have mass one. Typically only the Lebesgue measure $\theta$ is used as the base measure, but we will continue to describe the model in this greater generality where any $\Gamma$ can be used. The only assumption on the $\omega$ is that $e^{\lambda(\beta)} := \expect{e^{\beta \omega}} < \infty$ for all $\beta \in \R$. Clearly then the polymer measure can be expressed as a cascade measure with
\begin{align*}
d \Gamma_{\omega, \beta}^{(n)}(\xi) = \frac{e^{n \lambda(\beta)}}{Z_{\omega, \beta}^{(n)}} \, d \Gamma_{W_{\beta}}^{(n)}(\xi) = \frac{d \Gamma_{W_{\beta}}^{(n)}(\xi)}{\Gamma_{W_{\beta}}^{(n)}(\root)},
\end{align*}
with $W_{\beta}(v) = \exp \left \{ \beta \omega(v) - \lambda(\beta) \right \}$. If $\Gamma$ is $W_{\beta}$-regular then Section \ref{sec:Background} shows that the limiting polymer measure exists and is given by
\begin{align*}
\lim_{n \to \infty} d \Gamma_{\omega, \beta}^{(n)}(\xi) = \frac{d \Gamma_{W_{\beta}}(\xi)}{\Gamma_{W_{\beta}}(\root)}.
\end{align*}
If $\Gamma$ is not $W_{\beta}$-regular it is still an open problem as to whether or not a limit exists. Subsequential limits automatically exists because each finite level polymer measure is normalized to be a probability measure and the tree boundary $\partial \mt$ is compact, but the structure of the set of subsequential limits is not known. See \cite{WW10} for more on this problem.

Applying our cascade process to the study of polymer measures is most helpful whenever the family of cascading distributions $W_{\beta} = \exp \left \{ \beta \omega - \lambda(\beta) \right \}$ can be represented by a process $W_t$ satisfying Assumption \ref{assumption:basic}. By this we mean that the processes $W_{\beta}$ and $W_t$ have the same marginal distributions at fixed times (up to a possible change of variables between $\beta$ and $t$), but $W_t$ has the independent increments property of Assumption \ref{assumption:basic}. In this case, the cascade process of Section \ref{sec:process} gives us a coupling of the polymer measures at different temperatures that is different from the standard one obtained by simply multiplying the same variables by a different factor. The advantage of our coupling is that it has the Markov property implied by Section \ref{sec:markov_property}. In polymer language this Markov property has a nice interpretation: the polymer measure at a given temperature can be constructed by choosing a polymer at any higher temperature and then placing it in a new and independent environment. Most importantly, the higher temperature does \textit{not} have to be infinite.

The simplest case of a weight process satisfying the above is the Gaussian weights of Section \ref{sec:gaussian}. The scaling properties of Brownian motion imply that in this case the $t$ variable acts as both a time and an inverse temperature. This gives a nice interpretation to the stochastic calculus results of Proposition \ref{prop:total_mass_SDE}. The SDE for $\Gamma_t(\root)$ tells us that the total mass at the root evolves according to a weighted measure of the Brownian noise being inputted, with the weights prescribed by the polymer measure at the time infinitesimally beforehand. The formula for the quadratic variation tells us that it evolves according to the \textit{overlap} of the polymer measure, that is the expected amount of time that two polymers paths chosen independently under $\Gamma_t^*$ will spend together before eventually splitting. The explosion time of the cascade process is exactly when the accumulated overlap reaches infinity.

The Girsanov theory is also useful in this context. The tree polymer model can be thought of as a model of random walk in a random environment, where the random variables $\omega$ act as the environment. For this part we assume that $\Gamma = \theta$, and under the measure $\theta_{W_{\beta}}^*$ the process $\xi_0, \xi_1, \xi_2, \ldots$, is Markov with transition probabilities given by
\begin{align*}
\theta_{W_{\beta}}^* \left( \xi_{i+1} = (\xi_i)_L | \xi_0, \xi_1, \ldots, \xi_i \right) = \frac{\theta_{W_{\beta}}((\xi_i)_L)}{\theta_{W_{\beta}}(\xi_{i})}.
\end{align*}
To study this type of RWRE one typically uses the ``point of view of the particle'', which is the study of the environment Markov chain defined by
\begin{align*}
Z_n = \{ \omega(u) \}_{u \in \mt(\xi_n)}.
\end{align*}
Note that $Z_n$ takes values in the space of environments. It is straightforward to verify that if $Q$ is a measure under which the $\omega$ are iid random variables and $\xi$ is chosen according to the polymer measure $\theta_{W_{\beta}}^*$, then $Z_n$ is a stationary Markov process with the same transition probabilities as the $\xi_i$ Markov chain, i.e.
\begin{align*}
P \left( Z_{i+1} = \{ \omega(u) \}_{u \in \mt((\xi_i)_L)} | Z_0, \ldots, Z_i \right) = \theta_{W_{\beta}}^* \left( \xi_{i+1} = (\xi_i)_L | \xi_0, \xi_1, \ldots, \xi_i \right).
\end{align*}
See \cite{Z04} for more on the environment Markov chain. It begins in stationarity, with the stationary distribution being $\theta_{W_{\beta}}(\root) \, dQ(\omega)$. The Girsanov theory of Corollary \ref{corollary:girsanov} gives a way to analyze this stationary distribution. Assume that under $Q$ the $\omega$ are iid $N(0, T')$ for some $T' < 2 \log 2$. Then under $\theta_{\omega}(\root) \, dQ(\omega)$ the variables $\omega$ have the law of
\begin{align*}
\int_0^{T'} \frac{\theta_s(v)}{\theta_s(\root)} \, ds + \tilde{B}_{T'}(v),
\end{align*}
where the $\tilde{B}_t(v)$ are iid Brownian motions on the vertices of the tree. This gives an alternate description of the stationary measure for the environment Markov Chain.

\subsection{One-Dimensional Random Geometry and KPZ \label{sec:KPZ}}

Multiplicative cascades have also been used as a toy model for studies of random geometry, most notably in \cite{BS09}. There one considers the pushforward of $\Gamma_W$ onto the interval $[0,1]$ via binary expansion; left turns in $\xi$ correspond to zeros in the binary expansion and right turns to ones. We use $\Gamma_W$ to also denote the distribution function of the measure on $[0,1]$, i.e.
\begin{align*}
\Gamma_W(x) = \Gamma_W([0,x]).
\end{align*}
If $\Gamma_W$ is strictly positive, then $\Gamma_W(x)$ is a continuous, non-decreasing function on $[0,1]$. If $\Gamma_W(v) > 0$ for every $v \in \mt$, then $x \mapsto \Gamma_W(x)$ is strictly increasing and hence a continuous bijection of $[0,1]$ onto $[0, \Gamma_W(1)]$. In the case $\Gamma = \theta$, Benjamini and Schramm used this map to establish a relation between the Hausdorff dimension of a set and its random image under $\theta_W$. Specifically they show the following:

\begin{theorem}[\cite{BS09}]
Let $W$ be a cascading distribution with $\expect{W \log W} < \log 2$ (so that $\theta$ is $W$-regular), and assume that $\expect{W^{-s}} < \infty$ for all $s \in [0,1)$. Let $K \subset [0,1]$ be some non-empty, deterministic set. Then there is the following \textbf{KPZ formula}:
\begin{align*}
\hdim K = \phi_W \left( \hdim \theta_W(K) \right),
\end{align*}
where $\theta_W(K)$ is the (random) image of $K$ via the distribution function $\theta_W$, and $\phi_W$ is the deterministic bijection from $[0,1]$ onto $[0,1]$ given by
\begin{align*}
\phi_W(h) = h - \log_2 \expect{W^h}.
\end{align*}
\end{theorem}

Applying our process to this setup gives some interesting interpretations. Let $\theta_t$ and $\phi_t$ denote the corresponding cascade process and bijection when we replace $W$ by dynamic weights $W_t$. As time evolves, the image set $\theta_t(K)$ moves about the line and its Hausdorff changes with it, yet the dimension evolves deterministically even though the set evolves randomly. Remark \ref{remark:increasing_moments} and the formula above tell us that $\phi_t(h)$ is a decreasing function of $t$ for each fixed $h$, and hence Hausdorff dimensions get smaller as time evolves. Using our process it is possible to understand the infinitesimal evolution of the dimension. Indeed write $d(t) = \hdim \theta_t(K)$, and then the KPZ formula becomes
\begin{align*}
d(0) = \phi_t(d(t)).
\end{align*}
Differentiating both sides with respect to $t$ leads to an ODE for $d(t)$:
\begin{align*}
\dot{d}(t) = - \frac{\dot{\phi}_t(d(t))}{\phi_t'(d(t))}.
\end{align*}
The particulars of this ODE depends on the type of weight process being used. For example in the case of Gaussian weights as in Section \ref{sec:gaussian} it becomes
\begin{align*}
\dot{d} = - \frac{d(1-d)}{2 \log 2 - t(2d-1)} =: \psi_t(d).
\end{align*}
This ODE has many interesting aspects. First note that the $2 \log 2$ appears because it is the lifetime of the $\theta_t$ process, that is the time at which it collapses to the zero measure. Further, by the presence of the $t$ term in the denominator the ODE is non-autonomous, except at $d = 1/2$ where the non-autonomous term strangely disappears. It can also be shown that
\begin{align*}
\lim_{t \uparrow 2 \log 2} = 1 - \sqrt{1 - d(0)},
\end{align*}
so that even as $\theta_t$ approaches the zero measure the Hausdorff dimension of the random set stays bounded away from zero.

Although the work of Benjamini and Schramm can be used to derive the infinitesimal evolution of the Hausdorff dimension, in principle it should be possible to derive it separately and use it to give an alternate proof of their KPZ formula. All that needs to be found is a proof of the relation
\begin{align*}
\hdim \theta_{t+\delta}(K) = \hdim \theta_t(K) + \psi_t(\hdim \theta_t(K)) \delta + o(\delta)
\end{align*}
that does not use the Benjamini and Schramm statement (although many of the techniques of their proof would probably be incorporated), and then the Markov property of the $\theta_t$ process turns this infinitesimal relation at a fixed time into the ODE that holds at all times. We have attempted to derive this relation but thus far been unable to, although we hope a proof will be at hand soon. In fact we believe that there is a slightly more general fact lurking in the background: namely that if $\Gamma$ is an initial measure and $W$ a cascading distribution that is a small perturbation away from the degenerate distribution at one, then
\begin{align*}
\hdim \Gamma_W(K) = \hdim  \Gamma(K) + \psi_{\Gamma,W}(\hdim \Gamma(K)).
\end{align*}
Here $\psi_{\Gamma, W}$ would be a deterministic function determined by the properties of $\Gamma$ and the size and type of the perturbation of $W$ away from one. The infinitesimal relation is given by the ``derivative'' of $\psi$ as the cascading distribution concentrates at one. It is not clear to us exactly how the properties of $\Gamma$ enter into the picture, although we expect that they must in some form. It is also not clear if the relation above will be independent of the set $K$ for all initial measures $\Gamma$, although we expect it will be for initial measures with some type of self-similarity.

\bibliographystyle{alpha}
\bibliography{../StochasticMPBib}

\end{document}